\documentclass{article}
\usepackage{lipsum}
\usepackage[all]{xy}
\usepackage{amsmath, amsthm, latexsym, amssymb,
graphicx,color,cite,enumerate,mathrsfs,comment}
\usepackage{enumitem}
\newtheorem{theorem}{Theorem}[section]
\newtheorem{lemma}[theorem]{Lemma}
\newtheorem{definition}[theorem]{Definition}
\newtheorem{corollary}[theorem]{Corollary}
\newtheorem{proposition}[theorem]{Proposition}

\newtheorem{hypothesis}{Hypothesis}[section]

\newtheorem{remark}{Remark}

\newcommand*{\myproofname}{Proof}

\renewcommand\Re{\operatorname{Re}}

\newcommand\Exp{\operatorname{Exp}}
\newcommand\tr{\operatorname{tr}}

\newcommand\Spec{\operatorname{Spec}}
\renewcommand\det{\operatorname{det}}

\newcommand{\End}{\mbox{End}}
\newcommand{\Dom}{\mbox{\rm Dom}}
\newcommand{\Gl}{\mbox{Gl}}
\title{Davies' method for heat-kernel estimates: An extension to the semi-elliptic setting}
\author{Evan Randles and Laurent Saloff-Coste}

\begin{document}
\maketitle

\begin{abstract}

We consider a class of constant-coefficient partial differential operators on a finite-dimensional real vector space which exhibit a natural dilation invariance. Typically, these operators are anisotropic, allowing for different degrees in different directions.  The “heat” kernels associated to these so-called
positive-homogeneous operators are seen to arise naturally as the limits of
convolution powers of complex-valued measures, just as the classical heat kernel appears in the central limit theorem. Building on the functional-analytic approach developed by E. B. Davies for higher-order uniformly elliptic operators with measurable coefficients, we formulate a general theory for (anisotropic) self-adjoint variable-coefficient operators, each comparable to a positive-homogeneous operator, and study their associated heat kernels. Specifically, under three abstract hypotheses, we show that the heat kernels satisfy off-diagonal (Gaussian type) estimates involving the Legendre-Fenchel transform of the operator's principle symbol. Our results extend those of E. B. Davies and G. Barbatis and partially extend results of A. F. M. ter Elst and D. Robinson.
\end{abstract}

\noindent{\small\bf Mathematics Subject Classification:} Primary 35K08; Secondary 35K25, 35H30\\

\noindent{\small\bf Keywords:} Heat kernel estimates, Semi-elliptic operators, quasi-elliptic operators, Legendre-Fenchel transform.

\section{Introduction}

In \cite{Davies1995}, E. B. Davies develops an abstract method for establishing off-diagonal estimates for the heat kernels of self-adjoint uniformly elliptic higher-order partial differential operators on $\mathbb{R}^d$. In particular, Davies considers a general self-adjoint operator of the form
\begin{equation*}
Hf(x)=\sum_{|\alpha|,|\beta|\leq m}D^{\alpha}\left\{a_{\alpha,\beta}(x)D^{\beta}f(x)\right\}
\end{equation*}
and studies the corresponding ``heat'' kernel, $K_H$, of $H$ and its properties; here, $D^{\gamma}=(-i\partial_{x_1})^{\gamma_1}(-i\partial_{x_2})^{\gamma_2}\cdots(-i\partial_{x_d})^{\gamma_d}$ for each multi-index $\gamma$. Of course, when it exists, $K_H=K_H(t,x,y)$ is the integral kernel for the semigroup $\{e^{-tH}\}$ on $L^2$ generated by $H$ and is also recognized as the fundamental solution to the parabolic equation
\begin{equation*}
(\partial_t+H)u=0.
\end{equation*}
When $H$ is uniformly elliptic, i.e., $H$ is comparable to the $m$-th power of the Laplacian $(-\Delta)^m$, and under certain conditions discussed below, the method yields the estimate
\begin{equation}\label{eq:EllipEst}
|K_H(t,x,y)|\leq \frac{C_1}{t^{d/2m}}\exp\left(-tC_2\left|\frac{x-y}{t}\right|^{2m/(2m-1)}+Mt\right)
\end{equation}
for $t>0$, $x,y\in\mathbb{R}^d$, where $C_1,C_2$ and $M$ are positive constants. For the canonical case in which $H=(-\Delta)^m$, this estimate, with $M=0$, is readily established using an optimization argument and the Fourier transform. As discussed in \cite{Randles2017}, the optimization therein naturally selects the function $x\mapsto C_2|x|^{2m/(2m-1)}$ as the Ledendre-Fenchel transform of the symbol (or Fourier multiplier) $|\xi|^{2m}$ of the operator $(-\Delta)^m$. We encourage the reader to see the articles \cite{Randles2017}, \cite{Barbatis1996} and \cite{Blunck2005} for discussion of the appearance of the Legendre-Fenchel transform in heat kernel estimates. In the case that $H$ is a second-order operator, i.e., $m=1$, this is the well-studied Gaussian estimate \cite{Saloff-Coste2010}. The applications of estimates of the form \eqref{eq:EllipEst} are legion. In particular, \eqref{eq:EllipEst} guarantees that the semigroup $\{e^{-tH}\}$ extends to a strongly continuous semigroup $\{e^{-tH_p}\}$ on $L^p$ for all $1\leq p<\infty$ and moreover their generators, $H_p$, have spectra independent of $p$ \cite{Davies1995}.

In the case that the coefficients $\{a_{\alpha,\beta}(x)\}$  of $H$ are bounded and H\"{o}lder continuous, Levi's parametrix method, adapted to parabolic equations by A. Friedman and S. D. Eidelman, guarantees that a continuous heat kernel $K_H$ exists and satisfies the estimate \eqref{eq:EllipEst} \cite{Friedman1964,Eidelman1969}. When the coefficients $\{a_{\alpha,\beta}\}$ are merely bounded and measurable, Davies' method yields the estimate \eqref{eq:EllipEst} subject to a dimension-order restriction that $d/2m<1$. The restriction can be weakened to $d/2m\leq 1$ by the method of \cite{Auscher1998,terElst1997} but it cannot be weakened any further \cite{Davies1997a,deGiorgi1968,Mazya1968}. Specifically, for each integer $m$ such that $d/2m>1$, Davies \cite{Davies1997a} constructs a uniformly elliptic self-adjoint operator $H$ of order $m$ (which is a system when $d$ is odd) with bounded coefficients (in fact, smooth away from the origin) whose semigroup $\{e^{-tH}\}$ cannot be extended to a strongly continuous semigroup on $L^p$ for all $1\leq p<\infty$ and therefore the estimate \eqref{eq:EllipEst} cannot hold. Further discussion of this example can be found in \cite{Davies1997}. 

Moving beyond the elliptic (isotropic) setting, in this article, we introduce a class of constant-coefficient partial differential operators, which we call positive-homogeneous operators. Introduced in \cite{Randles2017}, these are hypoelliptic operators that interact well with certain dilations of the underlying space and they play the role that $(-\Delta)^m$ plays in the elliptic theory. We then consider a class of variable-coefficient operators, each comparable to a positive-homogeneous operator and study their associated heat kernels. We show that Davies' method, with suitable modification, carries over into our naturally anisotropic setting. 

To motivate our study, consider the constant-coefficient operator
\begin{equation*}
\Lambda=-\partial_{x_1}^2+\partial_{x_2}^4
\end{equation*}
on $\mathbb{R}^2$. Though this operator is not elliptic, it has many properties shared by elliptic operators. It is, for example, hypoelliptic; this can be seen by studying its symbol,
\begin{equation*}
R(\xi)=R(\xi_1,\xi_2)=\xi_1^2+\xi_2^4.
\end{equation*}
As $(-\Delta)^m$ plays well with (isotropic) dilations of $\mathbb{R}^d$, $\Lambda$ has the property that 
\begin{equation*}
t\Lambda=\delta_{1/t}\circ \Lambda\circ \delta_t
\end{equation*}
for all $t>0$ where $\delta_t(f)(x_1,x_2)=f(t^{1/2}x_1,t^{1/4}x_2)$ is given by the anisotropic dilation $(x_1,x_2)\mapsto (t^{1/2}x_1,t^{1/4}x_2)$ of $\mathbb{R}^2$; for this reason, $\Lambda$ is said to be homogeneous. As discussed in \cite{Randles2015a}, the homogeneity of $\Lambda$ is essential for the appearance of its heat kernel
\begin{equation*}
K_{\Lambda}(t,x,y)=\frac{1}{\sqrt{2\pi}}\int_{\mathbb{R}^2}e^{-i(x-y)\cdot\xi}e^{-tR(\xi)}\,d\xi,
\end{equation*}
defined for $t>0$ and $x,y\in\mathbb{R}^2$, as an attractor for convolution powers of complex-valued functions, i.e., its appearance in local limit theorems. An optimization argument, similar to that for $K_{(-\Delta)^m}$, gives the estimate
\begin{equation}\label{eq:TestHKEst}
|K_{\Lambda}(t,x,y)|\leq\frac{C_1}{t^{\omega_{\Lambda}}}\exp\left(-tC_2R^{\#}\left(\frac{x-y}{t}\right)\right)
\end{equation}
for $t>0$ and $x,y\in\mathbb{R}^2$ where
\begin{equation*}
R^{\#}(x)=R^{\#}(x_1,x_2)=\left(\frac{x_1}{2}\right)^2+3\left(\frac{x_2}{4}\right)^{4/3}
\end{equation*}
is the Legendre-Fenchel transform of $R$ and $\omega_{\Lambda}=1/2+1/4=3/4$ is known as the homogeneous order associated to $\Lambda$.  As we shall see, the homogeneous order $\omega_{\Lambda}$ depends on the order of derivatives appearing in $\Lambda$ and on the dimension of the underlying space; it generalizes the exponent $d/2m$ appearing in the prefactor in \eqref{eq:EllipEst} governing small-time on-diagonal decay.

By analogy to the theory of self-adjoint uniformly elliptic operators and their heat kernel estimates, we then ask: For a self-adjoint variable-coefficient operator $H$ which is comparable to a homogeneous operator $\Lambda$ with symbol $R$, under what conditions will the heat kernel for $H$ exists and satisfy an estimate of the form
\begin{equation*}
|K_{H}(t,x,y)|\leq \frac{C_1}{t^{\omega_{\Lambda}}}\exp\left(-tC_2R^{\#}\left(\frac{x-y}{t}\right)+Mt\right)?
\end{equation*}
It was shown in \cite{Randles2017}, using Levi's parametrix method adapted to our naturally anisotropic setting, that the above estimate is satisfied provided, in particular, that $H$ has H\"{o}lder continuous coefficients (see also \cite{Eidelman1960}). In this article, we extend these results to the realm in which $H$ has bounded measurable coefficients. To this end, we employ the abstract method of E. B. Davies which we modify in two ways. First, we adapt Davies' single-variable optimization procedure, which produces the term in the exponent of \eqref{eq:EllipEst}, to a multivariate optimization procedure suitably adapted to our anisotropic setting. In this way, we see the natural appearance of the Legendre-Fenchel transform. Our second modification to the theory allows for the dimension-order restriction $d/2m<1$ ($\omega_{\Lambda}<1$ in our case) to be lifted provided that certain integer powers of $H$ also behave well in perturbation estimates.


\section{Preliminaries}

As discussed in \cite{Randles2017}, to introduce the class of model operators considered in this article, it is useful to work in a framework which is coordinate-free. In view of the anisotropic nature of the problem we want to study, it is important to be free to choose coordinate systems adapted to each particular operator $\Lambda$ at hand. To this end, we consider a $d$-dimensional real vector space $\mathbb{V}$ equipped with the standard smooth structure; we do not affix $\mathbb{V}$ with a norm or basis.  The dual space of $\mathbb{V}$ is denoted by $\mathbb{V}^*$ and the dual pairing is denoted by $\xi(x)$ for $x\in\mathbb{V}$ and $\xi\in\mathbb{V}^*$. Let $dx$ and $d\xi$ be Haar (Lebesgue) measures on $\mathbb{V}$ and $\mathbb{V}^*$, respectively, which we take to be suitably normalized so that our conventions for the Fourier transform and inverse Fourier transform, given below, make each unitary. Throughout this article, all functions on $\mathbb{V}$ and $\mathbb{V}^*$ are understood to be complex-valued. Given a non-empty open set $\Omega\subseteq \mathbb{V}$, the usual Lebesgue spaces are denoted by $L^p(\Omega)=L^p(\Omega,dx)$ and equipped with their usual norms $\|\cdot\|_p$ for $1\leq p\leq \infty$. In the case that $p=2$, the corresponding inner product on $L^2(\Omega)$ is denoted by $\langle\cdot,\cdot\rangle$. Of course, we will also work with $L^2(\mathbb{V}^*)
:=L^2(\mathbb{V}^*,d\xi)$; here the $L^2$-norm and inner product will be denoted by $\|\cdot\|_{2^*}$ and $\langle\cdot,\cdot\rangle_*$ respectively. The Fourier transform $\mathcal{F}:L^2(\mathbb{V})\to L^2(\mathbb{V}^*)$ and inverse Fourier transform $\mathcal{F}^{-1}:L^2(\mathbb{V}^*)\to L^2(\mathbb{V})$ are defined, initially, for Schwartz functions $f\in \mathcal{S}(\mathbb{V})$ and $g\in\mathcal{S}(\mathbb{V}^*)$ by the formulas
\begin{eqnarray*}
\mathcal{F}(f)(\xi)=\hat{f}(\xi)=\int_{\mathbb{V}}e^{i\xi(x)}f(x)\,dx  && (\xi\in\mathbb{V}^*)
\end{eqnarray*}
and
\begin{eqnarray*}
\mathcal{F}^{-1}(g)(x)=\check{g}(x)=\int_{\mathbb{V}^*}e^{-i\xi(x)}g(\xi)\,d\xi  && (x\in\mathbb{V}).
\end{eqnarray*}

\noindent The symbols $\mathbb{R,C,Z}$ mean what they usually do; $\mathbb{N}$ denotes the set of non-negative integers. The symbols $\mathbb{R}_+$ and $\mathbb{N}_+$ denote the set of strictly positive elements of $\mathbb{R}$ and $\mathbb{N}$, respectively, and $\mathbb{C}_+$ denotes the set of complex numbers $z$ for which $\Re(z)>0$. Also, $\mathbb{R}_+^d$ and $\mathbb{N}_+^d$ respectively denote the set of $d$-tuples of $\mathbb{R}_+$ and $\mathbb{N}_+$. Adopting the summation notation for semi-elliptic operators presented in L. H\"{o}rmander's treatise \cite{Hormander1983}, for a fixed $\mathbf{m}=(m_1,m_2,\dots,m_d)\in\mathbb{N}_+^d$, we write
\begin{equation*}
|\beta:\mathbf{m}|=\sum_{k=1}^d\frac{\beta_k}{m_k}
\end{equation*} for all multi-indices $\beta=(\beta_1,\beta_2,\dots,\beta_d)\in\mathbb{N}^d$.\\

\noindent For the rest of this section, $W$ will denote a $d$-dimensional real vector space (meaning $\mathbb{V}$ or $\mathbb{V}^*$) and $\Omega$ will denote an open subset of $W$. The space of smooth functions on $\Omega$ is denoted by $C^\infty(\Omega)$ and the space of smooth functions with compact support in $\Omega$ is denoted by $C_0^\infty(\Omega)$. Taking $C_0^\infty(\Omega)$ to be equipped with its usual topology given by semi-norms, its dual space, the space of distributions, is denoted by $\mathcal{D}'(\Omega)$. Given $w\in W$, the derivation $D_{w}:\mathcal{D}'(\Omega)\to\mathcal{D}'(\Omega)$ is originally defined for $f\in C_0^\infty(\Omega)$ by the formula
\begin{equation*}
(D_wf)(x)=i\partial_wf(x)=i\left(\lim_{t\to 0}\frac{f(x+tw)-f(x)}{t}\right)
\end{equation*}
for $x\in \Omega$. Further, given a basis $\mathbf{w}=\{w_1,w_2,\dots,w_d\}$ of $W$, we introduce, for each multi-index $\beta\in \mathbb{N}^d$, the differential operator $D_{\mathbf{w}}^{\beta}:\mathcal{D}'(\Omega)\to\mathcal{D}'(\Omega)$ defined by
\begin{equation*}
D_{\mathbf{w}}^\beta=(D_{w_1})^{\beta_1}(D_{w_2})^{\beta_2}\cdots(D_{w_d})^{\beta_d}.
\end{equation*}
We shall denote by $\End(W)$ and $\Gl(W)$ the set of endomorphisms and isomorphisms of $W$ respectively. Given $E\in\End(W)$, we consider the one-parameter group $\{t^E\}_{t>0}\subseteq \Gl(W)$ defined by 
\begin{equation*}
t^E=\exp((\log t)E)=\sum_{k=0}^{\infty}\frac{(\log t)^k}{k!}E^k
\end{equation*}
for $t>0$.  These one-parameter subgroups of $\Gl(W)$ allow us to define continuous one-parameter groups of operators on the space of distributions as follows: Given $E\in\End(W)$ and $t>0$, first define $\delta_t^E(f)$ for $f\in C_0^{\infty}(W)$ by $\delta_t^E(f)(x)=f(t^Ex)$ for $x\in W$. Extending this to the space of distributions on $W$ in the usual way, the collection $\{\delta_t^E\}_{t>0}$ is a continuous one-parameter group of operators on $\mathcal{D}'(W)$. In the next section, we shall use these one-parameter groups to define a notion of homogeneity for partial differential operators. Given $\alpha=(\alpha_1,\alpha_2,\dots,\alpha_d)\in\mathbb{R}_+^d$ and a basis $\mathbf{w}=\{w_1,w_2,\dots,w_d\}$ of $W$, we denote by $E_{\mathbf{w}}^\alpha$ the isomorphism of $W$ defined by
\begin{equation}\label{eq:DefofE}
E_{\mathbf{w}}^\alpha w_k=\frac{1}{\alpha_k}w_k
\end{equation}
for $k=1,2,\dots, d$.\\

\noindent Finally, given a basis $\mathbf{w}=\{w_1,w_2,\dots,w_d\}$ of $W$, we define the map $\phi_{\mathbf{w}}:W\rightarrow\mathbb{R}^d$ by setting $\phi_{\mathbf{w}}(w)=(x_1,x_2,\dots,x_d)$ whenever $w=\sum_{l=1}^d x_l w_l$. This map defines a global coordinate system on $W$; any such coordinate system is said to be a linear coordinate system on $W$. By definition, a polynomial on $W$ is a function $P:W\rightarrow\mathbb{C}$ that is a polynomial function in some (and hence any) linear coordinate system on $W$. Of course, in the linear coordinate system defined by $\mathbf{w}$, each polynomial can be expressed as a linear combination of monomials of the form 
\begin{equation}\label{eq:Monomial}
w_{\mathbf{w}}^{\beta}=(x_1)^{\beta_1}(x_2)^{\beta_2}\cdots(x_d)^{\beta_d}
\end{equation}
where $\beta=(\beta_1,\beta_2,\dots,\beta_d)\in\mathbb{N}^d$ and $\phi_\mathbf{w}(w)=(x_1,x_2,\dots,x_d)$ as above. We say that a polynomial $P$ is positive-definite if its real part, $R=\Re P$, is non-negative and has $R(w)=0$ only when $w=0$. \\

\section{Homogeneous operators}\label{sec:HomogeneousOperators}
In this section we introduce a class of homogeneous constant-coefficient partial differential operators on $\mathbb{V}$. These operators will serve as ``model'' operators in our theory in the way that integer powers of the Laplacian serve a model operators in the elliptic theory of partial differential equations. To this end, let $\Lambda$ be a constant-coefficient partial differential operator on $\mathbb{V}$ and let $P:\mathbb{V}^*\rightarrow\mathbb{C}$ be its symbol. Specifically, $P$ is the polynomial on $\mathbb{V}^*$ defined by $P(\xi)=e^{-i\xi(x)}\Lambda(e^{i\xi(x)})$ for $\xi\in\mathbb{V}^*$ (this is independent of $x\in\mathbb{V}$ precisely because $\Lambda$ is a constant-coefficient operator). We first introduce the following notion of homogeneity of operators; it is mirrored by an analogous notion for symbols which we define shortly. 

\begin{definition}
Given $E\in\End(\mathbb{V})$, we say that a constant-coefficient partial differential operator $\Lambda$ is homogeneous with respect to the one-parameter group $\{\delta_t^E\}$ if
\begin{equation*}
\delta_{1/t}^E\circ \Lambda\circ \delta_t^E=t\Lambda
\end{equation*}
for all $t>0$; in this case we say that $E$ is a member of the exponent set of $\Lambda$ and write $E\in\Exp(\Lambda)$. 
\end{definition}

\noindent A constant-coefficient partial differential operator $\Lambda$ need not be homogeneous with respect to a unique one-parameter group $\{\delta_t^E\}$, i.e., $\Exp(\Lambda)$ is not necessarily a singleton. For instance, it is easily verified that, for the Laplacian $-\Delta$ on $\mathbb{R}^d$,
\begin{equation*}
\Exp(-\Delta)=2^{-1}I+\mathfrak{o}_d
\end{equation*}
where $I$ is the identity and $\mathfrak{o}_d$ is the Lie algebra of the orthogonal group, i.e., is given by the set of skew-symmetric matrices.\\

\noindent Given a constant coefficient operator $\Lambda$ with symbol $P$, one can quickly verify that $E\in\Exp(\Lambda)$ if and only if
\begin{equation}\label{eq:homofsymbol}
tP(\xi)=P(t^F\xi)
\end{equation}
for all $t>0$ and $\xi\in\mathbb{V}^*$ where $F=E^*$ is the adjoint of $E$. More generally, if $P$ is any continuous function on $W$ and \eqref{eq:homofsymbol} is satisfied for some $F\in\End(W)$, we say that $P$ \textit{is homogeneous with respect to} $\{t^F \}$ and write $F\in\Exp(P)$. This admitted slight abuse of notation should not cause confusion. In this language, we see that $E\in \Exp(\Lambda)$ if and only if $E^*\in\Exp(P)$.\\

\noindent We remark that the notion of homogeneity defined above is similar to that put forth for homogeneous operators on homogeneous (Lie) groups, e.g., Rockland operators \cite{Folland1982}. The difference is mostly a matter of perspective: A homogeneous group $G$ is equipped with a fixed dilation structure, i.e., it comes with a one-parameter group $\{\delta_t\}$, and homogeneity of operators is defined with respect to this fixed dilation structure. By contrast, we fix no dilation structure on $\mathbb{V}$ and formulate homogeneity in terms of an operator $\Lambda$ and the existence of a one-parameter group $\{\delta_t^E\}$ that plays well with $\Lambda$ in the sense defined above. As seen in the study of convolution powers on the square lattice (see \cite{Randles2015a}), it useful to have this freedom.

\begin{definition}\label{def:HomogeneousOperators}
Let $\Lambda$ be constant-coefficient partial differential operator on $\mathbb{V}$ with symbol $P$. We say that $\Lambda$ is a positive-homogeneous operator if $P$ is a positive-definite polynomial and $\Exp(\Lambda)$ contains a diagonalizable endomorphism.
\end{definition}

\noindent As discussed above, for a positive-homogeneous operator $\Lambda$, $\Exp(\Lambda)$ need not be a singleton. However, Lemma 2.10 of \cite{Randles2017} guarantees that, for any $E_1,E_2\in\Exp(\Lambda)$, 
\begin{equation*}
\tr E_1=\tr E_2.
\end{equation*}
Thus, to each positive-homogeneous operator $\Lambda$ we define the \textit{homogeneous order} of $\Lambda$ to be the number
\begin{equation*}
\mu_{\Lambda}=\tr E
\end{equation*}
for any $E\in\Exp(\Lambda)$. We note that the term ``homogeneous order'' does not coincide with the usual ``order" for partial differential operators. For instance, the Laplacian $-\Delta$ on $\mathbb{R}^d$ is a second-order operator, however, because $2^{-1}I\in \Exp(-\Delta)$, its homogeneous order is $\mu_{(-\Delta)}=\tr (2^{-1}I)=d/2$.\\

\noindent The proposition below shows, in particular, that every positive-homogeneous operator on $\mathbb{V}$ is semi-elliptic \cite{Browder1957, Hormander1983}  in some coordinate system. For a proof, see Section 2 of \cite{Randles2017}.

\begin{proposition}\label{prop:OperatorRepresentation}
Let $\Lambda$ be a positive-homogeneous operator on $\mathbb{V}$. Then there exist a basis $\mathbf{v}=\{v_1,v_2,\dots,v_d\}$ of $\mathbb{V}$ and $\mathbf{m}=(m_1,m_2,\dots,m_d)\in\mathbb{N}_+^d$ for which
\begin{equation}\label{eq:OperatorRepresentation1}
\Lambda=\sum_{|\beta:\mathbf{m}|=2}a_{\beta}D_{\mathbf{v}}^\beta. 
\end{equation}
where $\{a_{\beta}\}\subseteq\mathbb{C}$. The isomorphism $E_{\mathbf{v}}^{2\mathbf{m}}\in\Gl(\mathbb{V})$, defined by \eqref{eq:DefofE}, is a member of $\Exp(\Lambda)$ and therefore
\begin{equation*}
\mu_{\Lambda}=|\mathbf{1}:2\mathbf{m}|=\frac{1}{2m_1}+\frac{1}{2m_2}+\cdots+\frac{1}{2m_d}
\end{equation*}
where $\mathbf{1}:=(1,1,\dots,1)\in \mathbb{N}^d$. Furthermore, if $\mathbf{v}^*$ denotes the dual basis on $\mathbb{V}^*$ for the basis $\mathbf{v}$, 
\begin{equation*}
P(\xi)=\sum_{|\beta:\mathbf{m}|=2}a_\beta\xi^{\beta}
\end{equation*}
where $\xi^\beta=\xi_{\mathbf{v}^*}^\beta$ as in \eqref{eq:Monomial} and the isomorphism $E_{\mathbf{v}*}^{2\mathbf{m}}$ is a member of $\Exp(P)$.
\end{proposition}

\noindent We remark that, if a given positive-homogeneous operator $\Lambda$ is symmetric in the sense that $\langle \Lambda f,g\rangle=\langle f,\Lambda g\rangle$ for all $f,g\in C_0^{\infty}(\mathbb{V})$, then its symbol $P$ is necessarily real-valued, i.e., $R=\Re P=P$, and the coefficients $\{a_{\beta}\}$ of Proposition \ref{prop:OperatorRepresentation} are real numbers.\\

\section{Sobolev spaces, positive-homogeneous operators and their sesquilinear forms}

\noindent In the first part of this section, we define a family of Sobolev spaces on $\mathbb{V}$. These spaces, which include those of the classical elliptic theory, were also discussed in the context of $\mathbb{R}^d$ in \cite{Kannai1969} using coordinates. Then, given a symmetric positive-homogeneous operator $\Lambda$ on $\mathbb{V}$ with symbol $R$, we study the symmetric sesquilinear form $Q_{\Lambda}$ it defines. We then realize $\Lambda$ as a self-adjoint operator on $L^2$ whose domain and form domain are characterized by the previously defined Sobolev spaces; everything here relies on the semi-elliptic representation of positive-homogeneous operators given in Proposition \ref{prop:OperatorRepresentation}.\\

\noindent Let $1\leq p< \infty$, $\mathbf{m}\in \mathbb{N}_+^d$ and let $\mathbf{v}$ be a basis for $\mathbb{V}$. For a non-empty open set $\Omega\subseteq \mathbb{V}$, define
\begin{equation*}
W^{\mathbf{m},p}_{\mathbf{v}}(\Omega)=\left\{f\in L^p(\Omega):D_{\mathbf{v}}^\alpha f\in L^p(\Omega)\hspace{.1cm}\forall\hspace{.1cm}\alpha\mbox{ with }|\alpha:\mathbf{m}|\leq 1\right\}.
\end{equation*}
For any $f\in W^{\mathbf{m},p}_{\mathbf{v}}(\Omega)$ let
\begin{equation*}
\|f\|_{W_{\mathbf{v}}^{\mathbf{m},p}(\Omega)}=\left[\sum_{|\alpha:\mathbf{m}|\leq 1}\int_\Omega|D_{\mathbf{v}}^\alpha f|^pdx\right]^{1/p}.
\end{equation*}
Clearly, $\|\cdot\|_{W_{\mathbf{v}}^{\mathbf{m},p}(\Omega)}$ is a norm on $W^{\mathbf{m},p}_{\mathbf{v}}(\Omega)$ and the usual arguments show that $W^{\mathbf{m},p}_{\mathbf{v}}(\Omega)$ is a Banach space in this norm.  Naturally, we will call these spaces \textit{Sobolev spaces}; in the context of $\mathbb{R}^d$, these spaces were previously studied in \cite{Demidenko1993} and \cite{Kannai1969}.   Notice that when $\mathbb{V}=\mathbb{R}^d$, $\mathbf{v}=\mathbf{e}$ and $\mathbf{m}=(m,m,\dots,m)$, our definition coincides with that of $W^{m,p}(\Omega)$, the standard Sobolev spaces of $\mathbb{R}^d$ where, in this case, the basis is immaterial. Let us also denote by $W_{\mathbf{v},0}^{\mathbf{m},p}(\Omega)$ the closure of $C_0^{\infty}(\Omega)$ in the $\|\cdot\|_{W_{\mathbf{v}}^{\mathbf{m},p}}(\Omega)$ norm.\\

\noindent Temporarily, we restrict our attention to the case where $\Omega=\mathbb{V}$ and $p=2$. As one can check by the use of smooth cut-off functions and mollification, $C_0^{\infty}(\mathbb{V})$ is dense in $W_{\mathbf{v}}^{\mathbf{m},p}(\mathbb{V})$. The following result follows by the standard method, c.f., \cite{Lieb2001}; its proof is omitted.

\begin{lemma}\label{charsobolevbyfourierlem}
Let $\mathbf{m}\in\mathbb{N}^d$, $\mathbf{v}$ be a basis of $\mathbb{V}$ and $\mathbf{v}^*$ be the corresponding dual basis. Then
\begin{equation}\label{charsobolevlemeq}
W_{\mathbf{v}}^{\mathbf{m},2}(\mathbb{V})=\left\{f\in L^2(\mathbb{V}): \xi^{\alpha}\hat{f}(\xi)\in L^2(\mathbb{V}^*)\hspace{.1cm}\forall\hspace{.1cm}\alpha\mbox{ with }|\alpha:\mathbf{m}|\leq 1\right\}
\end{equation}
and
\begin{equation*}
\|f\|^2_{W_{\mathbf{v}}^{\mathbf{m},2}(\mathbb{V})}=\sum_{|\alpha:\mathbf{m}|\leq 1}\|\xi^{\alpha}\hat{f}(\xi)\|_{2^*}^2
\end{equation*}
where $\xi^{\alpha}=\xi_{\mathbf{v}^*}^\alpha$ as in \eqref{eq:Monomial}.
\end{lemma}

\begin{lemma}\label{charsobolevbyfourierlem2}
Let $\Lambda$ be a symmetric positive-homogeneous operator with symbol $R$ and, in view of Proposition \ref{prop:OperatorRepresentation}, let $\mathbf{m}\in\mathbb{N}_+^d$ and $\mathbf{v}$ be a basis of $\mathbb{V}$ as guaranteed by the proposition. 
Then
\begin{equation*}
W_{\mathbf{v}}^{\mathbf{m},2}(\mathbb{V})=\left\{f\in L^2(\mathbb{V}): \int_{\mathbb{V^*}}R(\xi)|\hat{f}(\xi)|^2d\xi<\infty\right\}
\end{equation*}
and moreover, the norms
\begin{equation*}
\|f\|':=\left(\|f\|_2^2+\int_{\mathbb{V^*}}R(\xi)|\hat{f}(\xi)|^2d\xi\right)^{1/2}
\end{equation*}
and $\|\cdot\|_{W_{\mathbf{v}}^{\mathbf{m},2}(\mathbb{V})}$ are equivalent.
\end{lemma}
\begin{proof}
By an appeal to Proposition \ref{prop:OperatorRepresentation} and Lemma \ref{lem:Scaling}, we obtain positive constants $C$ and $C'$ for which
\begin{equation*}
C(1+R(\xi))\leq \sum_{|\alpha:\mathbf{m}|\leq 1}\xi^{2\alpha}\leq C'( 1+R(\xi))
\end{equation*}
for all $\xi\in\mathbb{V}^*$. With this estimate, the result follows directly from Lemma \ref{charsobolevbyfourierlem} using the Fourier transform.
\end{proof}

\noindent Returning to the general situation, let $\Omega\subseteq \mathbb{V}$ be a non-empty open set. For $f\in L^2(\Omega)$, define $f_*\in L^2(\mathbb{V})$ by
\begin{equation}\label{eq:ExtensionDefinition}
f_*(x)=
\begin{cases}
f(x)&\mbox{ if }x\in\Omega\\
0&\mbox{ otherwise.}
\end{cases}
\end{equation}
Of course, $\|f\|_{L^2(\Omega)}=\|f_*\|_{L^2(\mathbb{V})}$. The following lemma shows that $W_{\mathbf{v},0}^{\mathbf{m},2}(\Omega)$ is continuously embedded in $W_{\mathbf{v}}^{\mathbf{m},2}(\mathbb{V})$:

\begin{lemma}\label{lem:SobolevEmbedding}
For any $f\in W_{\mathbf{v},0}^{\mathbf{m},2}(\Omega)$, $f_*\in W_{\mathbf{v}}^{\mathbf{m},2}(\mathbb{V})$ and
\begin{equation*}
\|f\|_{W_{\mathbf{v}}^{\mathbf{m},2}(\Omega)}=\|f_*\|_{W_{\mathbf{v}}^{\mathbf{m},2}(\mathbb{V})}.
\end{equation*}
\end{lemma}
\begin{proof}
Let $f\in W_{\mathbf{v},0}^{\mathbf{m},2}(\Omega)$ and let $\{f_n\}\subseteq C_0^{\infty}(\Omega)$ for which $\|f_n-f\|_{W_{\mathbf{v}}^{\mathbf{m},2}(\Omega)}\rightarrow 0$ as $n\rightarrow \infty$. Then for any $\phi\in C_0^{\infty}(\mathbb{V})$ and multi-index $\alpha$ for which $|\alpha:\mathbf{m}|\leq 1$,
\begin{multline*}
\int_{\mathbb{V}}f_*(D_{\mathbf{v}}^{\alpha}\phi) dx=\int_{\Omega}f (D_{\mathbf{v}}^{\alpha}\phi) dx=\lim_{n\rightarrow \infty}\int_{\Omega}f_n (D_{\mathbf{v}}^{\alpha}\phi) dx\\
=\lim_{n\rightarrow \infty}(-1)^{|\alpha|}\int_{\Omega}(D_{\mathbf{v}}^{\alpha} f_n)\phi dx=(-1)^{|\alpha|}\int_{\Omega} (D_{\mathbf{v}}^\alpha f )\phi dx\\
=(-1)^{|\alpha|}\int_{\mathbb{V}}(D_{\mathbf{v}}^{\alpha}f)_*\phi dx
\end{multline*}
where we used the fact that each $f_n$ has compact support in $\Omega$ and thus partial integration produces no boundary terms. Thus for each such $\alpha$, $D_{\mathbf{v}}^{\alpha}f_*=(D_{\mathbf{v}}^{\alpha}f)_*\in L^2(\mathbb{V})$ and $\|D_{\mathbf{v}}^{\alpha}f\|_{L^2(\Omega)}=\|D_{\mathbf{v}}^{\alpha}f_*\|_{L^2(\mathbb{V})}$ from which the result follows.
\end{proof}

\noindent We now turn to positive-homogeneous operators, viewed in the $L^2$ setting and their associated sesquilinear forms. Let $\Omega\subseteq \mathbb{V}$ be a non-empty open set and let $\Lambda$ be a positive-homogeneous operator on $\mathbb{V}$ with symbol $R$ and let $\mathbf{m}\in\mathbb{N}_+^d$ and $\mathbf{v}$ be the basis of $\mathbb{V}$ guaranteed by Proposition \ref{prop:OperatorRepresentation}. Define
\begin{equation*}
\Dom(Q_{\Lambda_\Omega})=W_{0,\mathbf{v}}^{\mathbf{m},2}(\Omega)
\end{equation*}
and for each $f,g\in \Dom(Q_{\Lambda_\Omega})$, put
\begin{equation*}
Q_{\Lambda_\Omega}(f,g)=\int_{\mathbb{V}^*}R (\xi)\widehat{f_*}(\xi)\overline{\widehat{g_*}(\xi)}d\xi.
\end{equation*}

\begin{proposition}\label{prop:DirichletOperator}
Then the restriction $\Lambda\vert_{C_0^{\infty}(\Omega)}$ extends to a non-negative self-adjoint operator on $L^2(\Omega)$, denoted by $\Lambda_{\Omega}$. Its associated symmetric sesquilinear form is $Q_{\Lambda_\Omega}$ and has $\Dom(Q_{\Lambda_\Omega})=W_{\mathbf{v},0}^{\mathbf{m},2}(\Omega)=\Dom(\Lambda_{\Omega}^{1/2})$. Moreover, $C_0^{\infty}(\Omega)$ is a core for $Q_{\Lambda_\Omega}$.
\end{proposition}
\begin{remark}
The self-adjoint operator $\Lambda_{\Omega}$ is the Dirichlet operator on $\Omega$, i.e., the operator associated with Dirichlet boundary conditions.
\end{remark}

\begin{proof}[Proof of Proposition \ref{prop:DirichletOperator}]
In view of Lemma \ref{charsobolevbyfourierlem2}, there are constants $C,C'>0$ for which
\begin{equation*}
C\|f\|_{W_{\mathbf{v}}^{\mathbf{m},2}(\mathbb{V})}\leq\left ( \|f\|_{L^2(\mathbb{V})}^2+\int_{\mathbb{V}^*}R(\xi)|\hat{f}(\xi)|^2d\xi\right)^{1/2}\leq C'\|f\|_{W_{\mathbf{v}}^{\mathbf{m},2}(\mathbb{V})}
\end{equation*}
for all $f\in W_{\mathbf{v}}^{\mathbf{m},2}(\mathbb{V})$. Thus by Lemma \ref{lem:SobolevEmbedding},
\begin{equation*}
C\|f\|_{W_{\mathbf{v}}^{\mathbf{m},2}(\Omega)}\leq \left(\|f\|_{L^2(\Omega)}^2+Q_{\Lambda_\Omega}(f)\right)^{1/2}\leq C'\|f\|_{W_{\mathbf{v}}^{\mathbf{m},2}(\Omega)}
\end{equation*}
for all $f\in W_{\mathbf{v},0}^{\mathbf{m},2}(\Omega)$. It follows that
\begin{equation*}
 \|f\|'_{\Omega}:=\left (\|f\|_{L^2(\Omega)}^2+Q_{\Lambda_\Omega}(f)\right)^{1/2}
\end{equation*}
defines a norm on $W_{\mathbf{v},0}^{\mathbf{m},2}(\Omega)$, equivalent to the norm $\|\cdot\|_{ W_{\mathbf{v}}^{\mathbf{m},2}(\Omega)}$. From this we can also conclude that $Q_{\Lambda_\Omega}$ is a \emph{bona fide} sesquilinear form with domain $\Dom(Q_{\Lambda_\Omega})=W_{\mathbf{v},0}^{\mathbf{m},2}(\Omega)$. 

In view of the positive-definiteness of $R$, it is easy to see that $Q_{\Lambda_\Omega}$ is symmetric, positive-definite (in the sense of forms) and densely defined. We claim that $Q_{\Lambda_\Omega}$ is closed. Indeed, let $\{f_n\}\subseteq  W_{\mathbf{v},0}^{\mathbf{m},2}(\Omega)$ be a $Q_{\Lambda_\Omega}$-Cauchy sequence and such that $f_n\rightarrow f$ in $L^2(\Omega)$ for some $f\in L^2(\Omega)$. Because the norms $\|\cdot\|'_{\Omega}$ and $\|\cdot\|_{ W_{\mathbf{v}}^{\mathbf{m},2}(\Omega)}$ are equivalent, we know that $\{f_n\}$ is also a Cauchy sequence in  $W_{\mathbf{v},0}^{\mathbf{m},2}(\Omega)$ and so it converges. Moreover, as the topology on $ W_{\mathbf{v},0}^{\mathbf{m},2}(\Omega)$ is finer than the topology induced by the $L^2(\Omega)$ norm, we can conclude that $f\in  W_{\mathbf{v},0}^{\mathbf{m},2}(\Omega)$ and $f_n\rightarrow f$ in $ W_{\mathbf{v},0}^{\mathbf{m},2}(\Omega)$. By again appealing to the equivalence 
of norms, it follows that $Q_{\Lambda_\Omega}$ is closed and, upon noting that $C_0^{\infty}(\Omega)$ is dense in $W_{\mathbf{v},0}^{\mathbf{m},2}(\Omega)$, it is evident that $C_0^{\infty}(\Omega)$ is a core for $Q_{\Lambda_\Omega}$.

In view of the theory of symmetric sesquilinear forms, $Q_{\Lambda_\Omega}$ has a unique associated non-negative self-adjoint operator $\Lambda_{\Omega}$ with $\Dom(\Lambda_{\Omega}^{1/2})=\Dom(Q_{\Lambda_\Omega})$.  Also, because
\begin{equation*}
\langle \Lambda f,g\rangle_{\Omega}=\langle \Lambda f_*,g_*\rangle=\int_{\mathbb{V}^*}R(\xi)\hat{f}_*(\xi)\overline{\hat{g}_*(\xi)}d\xi=Q_{\Lambda_\Omega}(f,g)=\langle f,\Lambda g\rangle_{\Omega}
\end{equation*}
for all $f,g\in C_0^{\infty}(\Omega)$, $\Lambda_{\Omega}$ must be a self-adjoint extension of $\Lambda\vert_{C_0^{\infty}(\Omega)}$.
\end{proof}

\begin{remark}
It should be pointed out that $\Lambda\vert_{C_0^{\infty}(\Omega)}$ is not generally essentially self-adjoint; for instance one can consider the Dirichlet and Neumann operators when $\Omega$ is, say, a bounded open non-empty subset of $\mathbb{V}$. 
\end{remark}

\noindent Our final proposition of this section addresses the essential self-adjointness of $\Lambda$ in the case that $\Omega=\mathbb{V}$. The proof is included for the convenience of the reader.
\begin{proposition}\label{prop:esa}
The operator $\Lambda\vert_{C_0^\infty(\mathbb{V})}$ is essentially self-adjoint and its closure $\Lambda=\Lambda_{\mathbb{V}}$ has
\begin{equation*}
\Dom(\Lambda)=W_{\mathbf{v}}^{2\mathbf{m},2}(\mathbb{V}).
\end{equation*}
\end{proposition}
\begin{proof}
We first show the essential self-adjointness of $\Lambda\vert_{C_0^\infty(\mathbb{V})}$. To this end, let $f\in \mbox{Ran}(\Lambda\vert_{C_0^{\infty}(\mathbb{V})}\pm i)^\perp$ and, in view of the unitarity of the Fourier transform, observe that
\begin{equation*}
0=\langle f,(\Lambda\pm i)g\rangle=\langle \hat{f},(R\pm i)\hat{g}\rangle_*=\langle (R\mp i)\hat{f},\hat{g}\rangle_*
\end{equation*}
for all $g\in C_0^{\infty}(\mathbb{V})$. We know that $\mathcal{F}(C_0^{\infty}(\mathbb{V}))$ is dense in $L^2(\mathbb{V}^*)$ and so it follows that $(R(\xi)\pm i)\hat{f}(\xi))=0$ almost everywhere. Using the fact that $R$ is real-valued, we conclude that $f=0$ and so $\mbox{Ran}(\Lambda\vert_{C_0^{\infty}(\mathbb{V})}\pm i)^\perp=\{0\}$. This implies that $\mbox{Ran}(\Lambda\vert_{C_0^{\infty}(\mathbb{V})}\pm i)$ is dense in $L^2(\mathbb{V})$ and thus $\Lambda\vert_{C_0^{\infty}(\mathbb{V})}$ is essentially self-adjoint in view of von Neumann's criteria. We denote this unique self-adjoint extension by $\Lambda$.

We now characterize the domain of $\Lambda$. Let $f\in \Dom(\Lambda)$ take a sequence $\{f_n\}\subseteq C^{\infty}_0(\mathbb{V})$ for which $f_n\to f$ and $\Lambda f_n\to \Lambda f$ in the sense of $L^2(\mathbb{V})$. For any multi-index $\alpha$ for which $|\alpha:2\mathbf{m}|\leq 1$, an appeal to Lemma \ref{lem:Scaling} gives a positive constant $C_{\alpha}$ for which 
\begin{equation*}
|\xi^{\alpha}|\leq C_{\alpha}(R(\xi)+1)
\end{equation*}
for all $\xi\in\mathbb{V}^*$ where $\xi^{\alpha}=\xi_{\mathbf{v}^*}^\alpha$ as in \eqref{eq:Monomial}. Consequently, for each pair of natural numbers $n$ and $m$,
\begin{eqnarray*}
\|D_{\mathbf{v}}^{\alpha}f_n-D_{\mathbf{v}}^{\alpha}f_m\|_2^2&=&\int_{\mathbb{V}}|D_{\mathbf{v}}^{\alpha}(f_n-f_m)(x)|^2\,dx\\
&=&\int_{\mathbb{V}^*}|\xi^{\alpha}(f_n-f_m)\hat{\,\,}(\xi)|^2\,d\xi\\
&\leq &C_{\alpha}^2\int_{\mathbb{V}^*}|(R(\xi)+1)(f_n-f_m)\hat{\,\,}(\xi)|^2\,d\xi\\
&\leq &C_\alpha^2\|(\Lambda+1)(f_n-f_m)\|_2^2
\end{eqnarray*}
where we have used the fact that $\{f_n\}\subseteq C_0^\infty(\mathbb{V})$. It now follows from the way the sequence $\{f_n\}$ was chosen that $\{D_{\mathbf{v}}^\alpha f_n\}$ is a Cauchy sequence in $L^2(\mathbb{V})$ and so it converges to some limit $g_\alpha$. Notice that, for each $\phi\in C_0^\infty(\mathbb{V})$, 
\begin{eqnarray*}
\lefteqn{\int_{\mathbb{V}}g_\alpha(x)\phi(x)\,dx=\lim_{n\to\infty}\int_{\mathbb{V}}D_{\mathbf{v}}^{\alpha}f_n(x)\phi(x)\,dx}\\
&=&\lim_{n\to\infty}(-1)^{|\alpha|}\int_{\mathbb{V}}f_n(x) D_{\mathbf{v}}^{\alpha}\phi(x)\,dx=(-1)^{|\alpha|}\int_{\mathbb{V}}f(x)D_{\mathbf{v}}^{\alpha}\phi(x)\,dx
\end{eqnarray*}
and thus $D_{\mathbf{v}}^{\alpha}f=g_\alpha\in L^2(\mathbb{V})$. Since this is true for each $\alpha$ such that $|\alpha:2\mathbf{m}|\leq 1$, we have $f\in W_{\mathbf{v}}^{2\mathbf{m},2}(\mathbb{V})$.

Conversely, let $f\in W_{\mathbf{v}}^{2\mathbf{m},2}(\mathbb{V})$ and, given the density of $C_0^\infty(\mathbb{V})$ in $W_{\mathbf{v}}^{2\mathbf{m},2}(\mathbb{V})$, let $\{f_n\}$ be a sequence of $C^\infty_0$ functions for which $f_n \to f$ in $W_{\mathbf{v}}^{2\mathbf{m},2}(\mathbb{V})$. Consequently, we have $D_{\mathbf{v}}^{\alpha}f_n \to D_{\mathbf{v}}^{\alpha}f$ in $L^2(\mathbb{V})$ for each multi-index $\alpha$ for which $|\alpha:2\mathbf{m}|\leq 1$. In particular, $f_n\to f$ and 
\begin{equation*}
\lim_{n\to\infty}\Lambda f_n=\lim_{n\to\infty}\sum_{|\alpha:\mathbf{m}|=2}a_{\alpha}D_{\mathbf{v}}^{\alpha}f_n=\sum_{|\alpha:\mathbf{m}|=2}a_{\alpha}D_{\mathbf{v}}^{\alpha}f
\end{equation*}
in $L^2(\mathbb{V})$. As $\Lambda$ is self-adjoint, it is closed and so necessarily $f\in \Dom(\Lambda)$. 
\end{proof}

\section{Ultracontractivity and Sobolev-type inequalities}

In this section we show that (self-adjoint) positive-homogeneous operators have many desirable properties shared by elliptic operators. In particular, for a self-adjoint positive-homogeneous operator $\Lambda$, we will prove corresponding Nash and Gagliardo-Nirenberg inequalities. \\

\noindent Let $\Lambda$ be a self-adjoint positive-homogeneous operator on $\mathbb{V}$ with symbol $R$ and homogeneous order $\mu_{\Lambda}$. In view of Proposition \ref{prop:DirichletOperator}, $\Lambda$ determines a self-adjoint positive-homogeneous operator on $L^2(\mathbb{V})$, $\Lambda_{\mathbb{V}}$. By an abuse of notation we shall write $\Lambda=\Lambda_{\mathbb{V}}$ and $Q_{\Lambda_{\mathbb{V}}}=Q_{\Lambda}$. Using the spectral calculus, define the semigroup $\{e^{-t\Lambda}\}$; this is a $C_0$-contraction semigroup of self-adjoint operators on $L^2(\mathbb{V})$. It should be no surprise that the semigroup $e^{-t\Lambda}$, defined here by the spectral calculus, coincides with that given by the Fourier transform; this, in particular, is verified by the following lemma.

\begin{lemma}\label{lem:Ultracontractivity}
For $f\in L^2(\mathbb{V})$ and $t>0$,
\begin{equation}\label{convolutionsemigroupeq}
\left(e^{-t\Lambda}f\right)(x)=\int_{\mathbb{V}}K_{\Lambda}(t,x-y)f(y)dy
\end{equation}
almost everywhere, where $K_{\Lambda}(t,x)=(e^{-tR})^{\vee}(x)\in\mathcal{S}(\mathbb{V})$. For each $t>0$, this formula extends $\{e^{-t\Lambda}\}$ to a bounded operator from $L^p(\mathbb{V})$ to $L^q(\mathbb{V})$ for any $1\leq p,q\leq \infty$. Furthermore, for each $1\leq p,q\leq\infty$, there exists $C_{p,q}>0$ such that
\begin{equation*}
\|e^{-t\Lambda}\|_{p\rightarrow q}\leq \frac{C_{p,q}}{t^{\mu_{\Lambda}(1/p-1/q)}}
\end{equation*}
for all $t>0$. In particular, the semigroup is ultracontractive with
\begin{equation*}
\|e^{-t\Lambda}\|_{2\rightarrow\infty}\leq \frac{C_{2,\infty}}{t^{\mu_{\Lambda}/2}}
\end{equation*}
for all $t>0$. 
\end{lemma}

\begin{remark}\label{rmk:Ultracontractivity}
A $C_0$-semigroup $\{T_t\}$ of self-adjoint operators on $L^2$ is said to be ultracontractive if, for each $t>0$, $T_t$ is a bounded operator from $L^2$ to $L^\infty$. We note that this condition immediately implies (by duality) that, for each $t>0$, $T_t$ is a bounded operator from $L^1$ to $L^\infty$ and this is often (though not exclusively, e.g., \cite{Gross1993}) taken to be the definition of ultracontractivity, see \cite{Coulhon1996}. Our terminology is not meant to imply (as it does in the case of Markovian semigroups) that the semigroup is contractive on $L^p$ for any $p$; it usually isn't.
\end{remark}

\begin{proof}[Proof of Lemma \ref{lem:Ultracontractivity}]
We first verify the representation formula \eqref{convolutionsemigroupeq}. Using the Fourier transform, one sees easily that convolution by $K_{\Lambda}$ defines a $C_0$-contraction semigroup on $L^2(\mathbb{V})$ of self-adjoint operators. Denote this semigroup and its corresponding generator by $T_t$ and $A$ respectively and note that $A$ is necessarily self-adjoint. For each $f\in C_0^{\infty}(\mathbb{V})$, observe that
\begin{equation*}
\lim_{t\rightarrow 0}\left\|t^{-1}\left(T_tf-f\right)+\Lambda f\right\|_2=\lim_{t\rightarrow 0}\left\|\left(t^{-1}(e^{-tR(\xi)}-1)+R(\xi)\right )\hat{f}(\xi)\right\|_{2^*}=0
\end{equation*}
where we have appealed to the dominated convergence theorem and the fact that $\mathcal{F}(\Lambda f)=R\hat{f}$. Consequently, $C_0^{\infty}(\mathbb{V})\subseteq\Dom(A)$ and $Af=-\Lambda f$ for all $f\in C_{0}^{\infty}(\mathbb{V})$. In view of Proposition \ref{prop:esa}, $\Lambda\vert_{C_0^{\infty}(\mathbb{V})}$ is essentially self-adjoint and so it must be the case that $A=-\Lambda$ and hence $T_t=e^{-\Lambda t}$ as claimed.

Finally, we establish the $L^p\rightarrow L^q$ estimates for $\{e^{-t\Lambda}\}$. In view of the representation \eqref{convolutionsemigroupeq} and Young's inequality for convolution,
\begin{equation*}
\|e^{-t\Lambda}\|_{p\to q}\leq \|K_{\Lambda}(t,\cdot)\|_s
\end{equation*}
where $1-\frac{1}{s}=\frac{1}{p}-\frac{1}{q}$. For $t>0$ and $E\in \Exp(\Lambda)$, we have
\begin{eqnarray*}
K_\Lambda(t,x)&=&\int_{\mathbb{V}^*}e^{-tR(\xi)}e^{-i\xi(x)}\,d\xi=\int_{\mathbb{V}^*}e^{-R(t^{E^*}\xi)}e^{-i\xi(x)}\,d\xi\\
&=&t^{-\tr E^*}\int_{\mathbb{V}^*}e^{-R(\xi)}e^{-i (t^{-E^*}\xi)(x)}\,dx\\
&=&t^{-\mu_{\Lambda}}K_{\Lambda}(1,t^{-E}x)
\end{eqnarray*}
for $x\in\mathbb{V}$ where we made a change of variables $\xi\mapsto t^{-E^*}\xi$. By making the analogous change of variables $x\mapsto t^Ex$, we obtain
\begin{eqnarray*}
\lefteqn{\|K_\Lambda(t,\cdot)\|_s=t^{-\mu_\Lambda}\|K_{\Lambda}(1,t^E(\cdot))\|_s}\\
&=&t^{-\mu_{\Lambda}+\mu_{\Lambda}/s}\|K_{\Lambda}(1,\cdot)\|_s=t^{-\mu_{\Lambda}(1/p-1/q)}\|K_{\Lambda}(1,\cdot)\|_s
\end{eqnarray*}
for $t>0$. The desired result follows by taking $C_{p,q}=\|K_{\Lambda}(1,\cdot)\|_s$ where $s=(1+1/q-1/p)^{-1}$.
\end{proof}

\begin{proposition}[Nash's inequality]
Let $\Omega$ be a non-empty open subset of $\mathbb{V}$ and let $\Lambda$ be a symmetric positive-homogeneous operator with homogeneous order $\mu_{\Lambda}$. We consider the self-adjoint operator $\Lambda_{\Omega}$ and its form $Q_{\Lambda_\Omega}$ given by Proposition \ref{prop:DirichletOperator}.  There exists $C>0$ such that
\begin{equation*}
\|f\|_{L^2(\Omega)}^{1+1/\mu_{\Lambda}}\leq C Q_{\Lambda_{\Omega}}(f)^{1/2}\|f\|_{L^1(\Omega)}^{1/\mu_{\Lambda}}
\end{equation*}
for all $f\in \Dom(Q_{\Lambda_{\Omega}})\cap L^1(\Omega)$.
\end{proposition}
\begin{proof}
It suffices to prove the estimate when $\Omega=\mathbb{V}$, for the general result follows from the isometric embedding of $W_{\mathbf{v},0}^{\mathbf{m},2}(\Omega)$ into $W_{\mathbf{v}}^{\mathbf{m},2}(\mathbb{V})$, c.f., Lemma \ref{lem:SobolevEmbedding}, and that of $L^1(\Omega)$ into $L^1(\mathbb{V})$. Again, we will denote $\Lambda_{\mathbb{V}}$ and $Q_{\Lambda_\mathbb{V}}$ by $\Lambda$ and $Q_\Lambda$ respectively. In view of Lemma \ref{lem:Ultracontractivity}, the self-adjointness of $\Lambda$ and duality give $C'>0$ such that
\begin{equation*}
\|e^{-t\Lambda}\|_{1\rightarrow 2}\leq \frac{C'}{t^{\mu_{\Lambda}/2}}
\end{equation*}
for all $t>0$. Thus for any $f\in \Dom(Q_{\Lambda})\cap L^1(\mathbb{V})$,
\begin{eqnarray}\label{ultraimplynasheq}\nonumber
\|f\|_2&\leq&\|e^{-t\Lambda}f-f\|_2+\|e^{-t\Lambda}f\|_2\\\nonumber
&\leq&\left\|\int_0^{t}\frac{d}{ds}e^{-s\Lambda}fds\right\|_2+\frac{C'}{t^{\mu_{\Lambda}/2}}\|f\|_1\\\nonumber
&\leq&\int_{0}^{t}\|\Lambda^{1/2}e^{-s\Lambda}\Lambda^{1/2}f\|_2ds+\frac{C'}{t^{\mu_{\Lambda}/2}}\|f\|_1\\
&\leq& \int_0^t\|\Lambda^{1/2}e^{-s\Lambda}\|_{2\rightarrow 2}dsQ_{\Lambda}(f)^{1/2}+\frac{C'}{t^{\mu_{\Lambda}/2}}\|f\|_1\\\nonumber
\end{eqnarray}
for all $t>0$. By virtue of the spectral theorem, we have
\begin{equation*}
\|\Lambda^{1/2}e^{-s\Lambda}\|_{2\rightarrow 2}\leq \sup_{\lambda>0}|\lambda^{1/2}e^{-s\lambda}|\leq \frac{C''}{s^{1/2}}
\end{equation*}
for all $s>0$ and therefore
\begin{equation*}
\|f\|_2\leq 2C''t^{1/2}Q_{\Lambda}(f)^{1/2}+\frac{C'}{t^{\mu_{\Lambda}}}\|f\|_1
\end{equation*}
for all $t>0$. The result follows by optimizing the above inequality and noting that $\mu_{\Lambda}>0$. 
\end{proof}
\noindent Suppose additionally that $\mu_{\Lambda}<1$. Using ultracontractivity directly, a calculation analogous to \eqref{ultraimplynasheq} yields
\begin{eqnarray*}
\|f\|_{\infty}&\leq&\int_0^t\|e^{-s\Lambda/2}\|_{2\rightarrow\infty}\|\Lambda^{1/2}e^{-s\Lambda/2}\|_{2\rightarrow 2}Q_{\Lambda}(f)^{1/2}\,ds+\frac{C}{t^{\mu_{\Lambda}/2}}\|f\|_2\\
&\leq&C't^{(1-\mu_{\Lambda})/2}Q_{\Lambda}(f)^{1/2}+\frac{C}{t^{\mu_{\Lambda}/2}}\|f\|_2\\
\end{eqnarray*}
for $f\in C_0^{\infty}(\Omega)$ and $t>0$. Upon optimizing with respect to $t$ and using the density of $C_0^{\infty}(\Omega)$ in $W_{\mathbf{v},0}^{\mathbf{m},2}(\Omega)$, we obtain the following lemma:
 
\begin{lemma}\label{nashlikelem}
If $\mu_{\Lambda}<1$ then there is $C>0$ such that for all $f\in W_{\mathbf{v},0}^{\mathbf{m},2}(\Omega)$, $f\in L^\infty(\Omega)$ and
\begin{equation*}
\|f\|_{L^\infty(\Omega)}\leq CQ_{\Lambda_{\Omega}}(f)^{\mu_{\Lambda}/2}\|f\|_{L^2(\Omega)}^{1-\mu_{\Lambda}}.
\end{equation*}
\end{lemma}

\noindent Lemma \ref{nashlikelem} is the analog of the Gagliardo-Nirenberg inequality in our setting.

\section{Fundamental Hypotheses}\label{sec:FundamentalHypotheses}

Let $\Omega$ be a non-empty open subset of $\mathbb{V}$. In this section, we will introduce three hypotheses concerning a symmetric sesquilinear form $Q$ (also called Hermitian form) defined on $C_0^\infty(\Omega)$ viewed as a subspace of the Hilbert space $L^2(\Omega)$. The first hypothesis will guarantee that the form is closable and its closure is associated to a self-adjoint operator $H$ on $L^2(\Omega)$. It is under these hypotheses that we will be able to establish the existence of the heat kernel for $H$ and prove corresponding off-diagonal estimates. Our construction is based on E. B. Davies' article \cite{Davies1995}, wherein a general class of higher order self-adjoint uniformly elliptic operators on $\mathbb{R}^d$ is studied. In what follows (and for the next three sections) $\|\cdot\|_2$ denotes the $L^2(\Omega)$ norm, $\langle\cdot,\cdot\rangle$ denotes its inner product. All mentions of a positive-homogeneous operator $\Lambda$ refer to the self-adjoint operator $\Lambda_{\Omega}$ of Proposition \ref{prop:DirichletOperator}. Correspondingly, $Q_{\Lambda_{\Omega}}$ is denoted by $Q_{\Lambda}$.\\

\begin{hypothesis}\label{hyp:Garding}
Let $Q$ be as above. There exists a self-adjoint positive-homogeneous operator $\Lambda$ with corresponding symmetric sesquilinear form $Q_{\Lambda}$ such that
\begin{equation}\label{eq:Garding}
\frac{1}{2}Q_{\Lambda}(f)\leq Q(f)\leq C(Q_{\Lambda}(f)+||f||_2^2) 
\end{equation}
for all $f\in C_0^{\infty}(\Omega)$ where $C\geq 1$. 
\end{hypothesis}

\noindent As noticed above, Hypothesis \ref{hyp:Garding} guarantees that $Q$ is bounded below and therefore closable. Its closure, which we still denote by $Q$, defines uniquely a self-adjoint operator $H$; we refer to $H$ as the operator associated to $Q$. Hypothesis \ref{hyp:Garding} is a comparability statement between $H$ and the positive-homogeneous operator $\Lambda$; for this reason, we say that $\Lambda$ is a \emph{reference operator} for $H$ (and for $Q$). In this way, \eqref{eq:Garding} is analogous to G\r{a}rding's inequality in that the latter compares second-order elliptic operators to the Laplacian.

\begin{remark} Necessarily, $C_0^\infty(\Omega)$ is a core for $Q$ and we have 
\begin{equation*}\Dom(H)\cup C_0^\infty(\Omega)\subseteq \Dom(Q)\subseteq L^2(\Omega).
\end{equation*}
It may however be the case that $\Dom(H)\cap C_0^\infty(\Omega)=\{0\}$, c.f., \cite{Davies1997}.
 \end{remark}

\noindent The inequality $\eqref{eq:Garding}$ further ensures that $\Dom(Q)=\Dom(Q_{\Lambda})$ and that $H\geq 0$. In view of Proposition \ref{prop:DirichletOperator}, there exist $\mathbf{m}\in\mathbb{N}^d$ and a basis $\mathbf{v}$ of $\mathbb{V}$ such that 
\begin{equation*}
\Dom(Q)=\Dom(Q_{\Lambda})=W_{\mathbf{v},0}^{\mathbf{m},2}(\Omega) 
\end{equation*}
and, because $C_0^\infty(\Omega)$ is dense in $W_{\mathbf{v},0}^{\mathbf{m},2}(\Omega)$, \eqref{eq:Garding} holds for all $f$ in this common domain. These remarks are summarized in the following lemma:

\begin{lemma}\label{gardingsobolevlem}
Let $Q$ satisfy Hypothesis \ref{hyp:Garding} with reference operator $\Lambda$. The associated operator $H$ is non-negative and
\begin{equation*}
\Dom(Q)=W_{\mathbf{v},0}^{\mathbf{m},2}(\Omega)
\end{equation*}
where $\mathbf{m}$ and $\mathbf{v}$ are those associated with $\Lambda$ via Proposition \ref{prop:DirichletOperator}. Moreover, \eqref{eq:Garding} holds for all $f$ in this common domain.
\end{lemma}

\noindent In view of the preceding lemma, any future reference to a sesquilinear form $Q$ which satisfies Hypothesis \ref{hyp:Garding} with reference operator $\Lambda$ is a reference to the closed form $Q$ whose domain is characterized by Lemma \ref{gardingsobolevlem} and has associated self-adjoint operator $H$. For the most part, as is done in \cite{Davies1995}, we will avoid identifing $\Dom(H)$ as it generally won't be necessary. By virtue of Lemma \ref{gardingsobolevlem} and Theorem 1.53 of \cite{Ouhabaz2009}, $-H$ generates a strongly continuous semigroup $T_t=e^{-tH}$ on $L^2(\Omega)$ which is a bounded holomorphic semigroup on a non-trivial sector of $\mathbb{C}$. The main goal of this article is to show that the semigroup $T_t$ has an integral kernel $K_H$ satisfying off-diagonal estimates in terms of the Legendre-Fenchel transform of $R$; we refer the reader to Section 3 of \cite{Randles2017} and Appendix \ref{Appendix:LF} of this article for the definition and useful properties of the Legendre-Fenchel transform of $R$. Under the hypotheses given in this section, we obtain these off-diagonal estimates by means of Davies' perturbation method, suitably adapted to our naturally anisotropic setting. Specifically, we study perturbations of the semigroup $T_t$ formed by conjugating $T_t$ by ``nice" operators. Denoting by $C^\infty(\Omega,\Omega)$ the set of smooth functions mapping $\Omega$ into itself, we set
\begin{equation*}
C_{\infty}^{\infty}(\Omega,\Omega)=\{\phi\in C^{\infty}(\Omega,\Omega): \partial_v^k(\lambda(\phi))\in L^{\infty}(\Omega)\\,\forall \,v\in\mathbb{V},\lambda\in\mathbb{V}^*\mbox{ and }k\geq 0\}.
\end{equation*}
Given $\phi\in C_{\infty}^{\infty}(\Omega,\Omega)$ and $\lambda\in\mathbb{V}^*$, we consider the smooth functions $e^{\lambda(\phi)}$ and $e^{-\lambda(\phi)}$; these will act as bounded and real-valued multiplication operators on $L^2(\Omega)$. For each such $\lambda$ and $\phi$, we define the \textit{twisted} semigroup $T^{\lambda,\phi}_t$ on $L^2(\Omega)$ by
\begin{equation*}
T_t^{\lambda,\phi}=e^{\lambda(\phi)}T_te^{-\lambda(\phi)}
\end{equation*}
for $t>0$. For any $f\in L^2(\Omega)$ such that $e^{-\lambda(\phi)}f\in \Dom(H)$, observe that
\begin{eqnarray*}
e^{\lambda(\phi)}(-H)e^{-\lambda(\phi)}f&=&e^{\lambda(\phi)}\lim_{t\rightarrow 0}\frac{T_t(e^{-\lambda(\phi)}f)-(e^{-\lambda(\phi)}f)}{t}\\
&=&\lim_{t\rightarrow 0}\frac{T^{\lambda,\phi}_tf-f}{t}
\end{eqnarray*}
where we have used the fact that $e^{\lambda(\phi)}$ acts as a bounded multiplication operator on $L^2(\Omega)$. Upon pushing this argument a little further one sees that $T_t^{\lambda,\phi}$ has infinitesimal generator $-H_{\lambda,\phi}=-e^{\lambda(\phi)}He^{-\lambda(\phi)}=e^{\lambda(\phi)}(-H)e^{-\lambda(\phi)}$ and 
\begin{equation*}
\Dom(H_{\lambda,\phi})=\left\{f\in L^2(\Omega): e^{-\lambda(\phi)}f\in\Dom(H)\right\}.
\end{equation*}
We also note that, in view of the resolvent characterization of bounded holomorphic semigroups, e.g., Theorem 1.45 of \cite{Ouhabaz2009}, it is straightforward to verify that $\{T_t^{\lambda,\phi}\}$ is a bounded holomorphic semigroup on $L^2(\Omega)$.
\begin{remark}
This construction for $T^{\lambda,\phi}_t$ is similar to that done in \cite{Davies1995}. The difference being that $\lambda$ for us is a ``multi-parameter'' whereas in \cite{Davies1995} it is a scalar. This construction is the basis behind the suitable adaptation of Davies' method for positive-homogeneous operators, discussed in the introductory section of this article. 
\end{remark}

\noindent In the same spirit, define \textit{twisted} form $Q_{\lambda,\phi}$ by
\begin{equation*}
Q_{\lambda,\phi}(f,g)=Q(e^{-\lambda(\phi)}f,e^{\lambda(\phi)}g)
\end{equation*}
for all $f,g\in \Dom(Q_{\lambda,\phi}):=\Dom(Q)$. This definition is meaningful because multiplication by $e^{\pm\lambda(\phi)}$ is continuous on $\Dom(Q)=W_{\mathbf{v},0}^{\mathbf{m},2}(\Omega)$. As usual, we write $Q_{\lambda,\phi}(f)=Q_{\lambda,\phi}(f,f)$ for $f\in\Dom(Q_{\lambda,\phi})$ and we note that $Q_{\lambda,\phi}$ isn't symmetric or real-valued. As the next lemma shows, $H_{\lambda,\phi}$ corresponds to $Q_{\lambda,\phi}$ in the usual sense.

\begin{lemma}\label{formgeneratorlambdaphilem}
For any $\lambda\in\mathbb{V}^*$ and $\phi\in C_{\infty}^{\infty}(\Omega,\Omega)$,
\begin{equation*}
\Dom(H_{\lambda,\phi})\subseteq \Dom(Q_{\lambda,\phi})=\Dom(Q)
\end{equation*}
and
\begin{equation*}
Q_{\lambda,\phi}(f)=\langle H_{\lambda,\phi}f,f\rangle
\end{equation*}
for all $f\in\Dom(H_{\lambda,\phi})$.
\end{lemma}
\begin{proof}
For $f\in\Dom(H_{\lambda,\phi})$,
\begin{equation*}
 e^{-\lambda(\phi)}f\in \Dom(H)\subseteq\Dom(Q)=W_{\mathbf{v},0}^{\mathbf{m},2}(\Omega).
\end{equation*}
Because $\phi\in C_{\infty}^{\infty}(\Omega,\Omega)$, $\partial_{v_i}^ke^{\lambda(\phi)}\in L^{\infty}(\Omega)$ for all $i=1,2,\dots, d$ and $k\geq 0$ . Using the Leibniz rule it follows that
\begin{equation*}
f=e^{\lambda(\phi)}(e^{-\lambda(\phi)}f)\in W_{\mathbf{v},0}^{\mathbf{m},2}(\Omega)=\Dom(Q_{\lambda,\phi}).
\end{equation*}
We see that,
\begin{equation*}
\langle H_{\lambda,\phi}f,f\rangle=\langle H(e^{-\lambda(\phi)}f),e^{\lambda(\phi)}f\rangle=Q(e^{-\lambda(\phi)}f,e^{\lambda(\phi)}f)=Q_{\lambda,\phi}(f)
\end{equation*}
as desired.
\end{proof}

\noindent Our second fundamental hypothesis is as follows:

\begin{hypothesis}\label{hyp:FormCompare}

Let $Q$ satisfy Hypothesis \ref{hyp:Garding} with reference operator $\Lambda$. There exist $\mathcal{E}\subseteq C_{\infty}^{\infty}(\Omega,\Omega)$ and $M>0$ such that:
\begin{enumerate}[label=\roman*]
\item For each pair $x,y\in\Omega$, there is $\phi\in \mathcal{E}$ for which $\phi(x)-\phi(y)=x-y$.
\item For all $\phi\in\mathcal{E}$, $\lambda\in\mathbb{V}^*$ and $f\in\Dom(Q)$,
\begin{equation}\label{eq:FormCompare1}
|Q_{\lambda,\phi}(f)-Q(f)|\leq \frac{1}{4}(Q(f)+M(1+R(\lambda))\|f\|_2^2)
\end{equation}
\end{enumerate}
where $R$ is the symbol of $\Lambda$. We will call \eqref{eq:FormCompare1} the form comparison inequality.
\end{hypothesis}

\noindent Our next lemma follows immediately from Lemma \ref{formgeneratorlambdaphilem} and Hypothesis \ref{hyp:FormCompare}. Its proof is omitted.
\begin{lemma}\label{lowboundfortwistedformlem}
Let $\phi\in\mathcal{E}$ and $\lambda\in\mathbb{V}^*$. If Hypothesis \ref{hyp:FormCompare} holds,
\begin{equation}\label{Htwistedlemmaeq}
2\Re[Q_{\lambda,\phi}(f)]=2\Re[(H_{\lambda,\phi}f,f)]\geq -\frac{M}{2}(1+R(\lambda))\|f\|_2^2
\end{equation}
for all $f\in\Dom(H_{\lambda,\phi})$.
\end{lemma}

\noindent Our final hypothesis is more technical and involves a perturbation estimate for sufficiently high powers of $H$, the self-adjoint operator associated to $Q$. Whereas Hypothesis \ref{hyp:Garding} and \ref{hyp:FormCompare} are easily satisfied, the third hypothesis is much more subtle, difficult to verify and restrictive.

\begin{hypothesis}\label{hyp:kappa}
Let $Q$ satisfy Hypotheses \ref{hyp:Garding} and \ref{hyp:FormCompare} with reference operator $\Lambda$ and associated self-adjoint operator $H$. Further, let $R$ be the symbol and $\mu_{\Lambda}$ be the homogeneous order of $\Lambda$, respectively. Set $\kappa=\min\{n\in\mathbb{N}:\mu_{\Lambda}/n<1\}$ and denote by $Q_{\Lambda^{\kappa}}$ the sesquilinear form corresponding to $\Lambda^{\kappa}$. There is $C>0$ such that, for any $\phi\in\mathcal{E}$ and $\lambda\in\mathbb{V}^*$,
\begin{equation*}
\Dom(H^{\kappa}_{\lambda,\phi})\subseteq \Dom(Q_{\Lambda^{\kappa}})
\end{equation*}
and 
\begin{equation*}
Q_{\Lambda^{\kappa}}(f)\leq C(|\langle H_{\lambda,\phi}^{\kappa}f,f\rangle|+(1+R(\lambda))^{\kappa}\|f\|_2^2)
\end{equation*}
for all $f\in \Dom(H_{\lambda,\phi}^{\kappa})$.
\end{hypothesis}

\noindent In \cite{Davies1995}, the self-adjoint operators considered are required to satisfy Hypothesis \ref{hyp:Garding} in the special case that $\Lambda=(-\Delta)^m$ on $\mathbb{R}^d$ for some $m\in\mathbb{N}$. The theory in \cite{Davies1995} proceeds under only two hypotheses which are paralleled by Hypotheses \ref{hyp:Garding} and \ref{hyp:FormCompare} above respectively. Incidentally, off-diagonal estimates are only shown in the case that $2m<d$ which corresponds to $\mu_{\Lambda}<1$ in our setting. As the proposition below shows, when $\mu_{\Lambda}<1$, Hypothesis \ref{hyp:kappa} is superfluous.

\begin{proposition}\label{prop:kappa}
Let $Q$ satisfy Hypotheses \ref{hyp:Garding} and \ref{hyp:FormCompare} with reference operator $\Lambda$ and associated self-adjoint operator $H$. Let $\mu_{\Lambda}$ be the homogeneous order of $\Lambda$. If $\mu_{\Lambda}<1$, i.e., $\kappa=1$, then Hypothesis \ref{hyp:kappa} holds.
\end{proposition}
\begin{proof}
The assertion that $\Dom(H_{\lambda,\phi})\subseteq \Dom(Q_{\Lambda})$ for all $\phi\in\mathcal{E}$ and $\lambda\in\mathbb{V}^*$ is a consequence of Lemma \ref{formgeneratorlambdaphilem}. Using \eqref{eq:Garding} and \eqref{eq:FormCompare1}, we have
\begin{eqnarray*}
Q_{\Lambda}(f)\leq 2Q(f)&\leq& C(\Re( Q_{\lambda,\phi}(f))+(1+R(\lambda))\|f\|_2^2)\\
&\leq& C(|Q_{\lambda,\phi}(f)|+(1+R(\lambda))\|f\|_2^2)
\end{eqnarray*}
for all $f\in \Dom(Q)$, $\phi\in\mathcal{E}$ and $\lambda\in\mathbb{V}^*$. In view of Lemma \ref{formgeneratorlambdaphilem}, the proof is complete. 
\end{proof}

\section{The $L^2$ theory}

We now return to the general theory. Throughout this section all hypotheses are to include Hypotheses \ref{hyp:Garding} and \ref{hyp:FormCompare} without explicit mention. With the exception of Lemma \ref{Hkappalemma}, all statements mirror those in \cite{Davies1995} and their proofs follow with little or no change. We will keep track of certain constants and to this end, any mention of $M>0$ refers to that which is specified in Hypothesis \ref{hyp:FormCompare}. Positive constants denoted by $C$ will change from line to line.

\begin{lemma}\label{twistedsgboundlemma} 
For any $\lambda\in\mathbb{V}^*$ and $\phi\in\mathcal{E}$,
\begin{equation*}
\|T_t^{\lambda,\phi}\|_{2\rightarrow 2}\leq \exp(M(1+R(\lambda))t/4)
\end{equation*}
for all $t>0$.
\end{lemma}
\begin{proof}
For $f\in L^2(\Omega)$, put $f_t=T_t^{\lambda,\phi}f$. By Lemma \ref{lowboundfortwistedformlem},
\begin{equation*}
\frac{d}{dt}\|f_t\|_2^2=-2\Re[(H_{\lambda,\phi}f_t,f_t)]\leq\frac{M}{2}(1+R(\lambda))\|f_t\|_2^2.
\end{equation*}
The result now follows from Gr\"{o}nwall's lemma.
\end{proof}

\begin{lemma}\label{twistedgenandsgboundlemma}
There exists $C>0$ such that
\begin{equation*}
\|H_{\lambda,\phi}T_t^{\lambda,\phi}\|_{2\rightarrow 2}\leq\frac{C}{t}\exp\left(\frac{M}{2}(1+R(\lambda))t\right)
\end{equation*}
for all $t>0$, $\lambda\in\mathbb{V}^*$ and $\phi\in\mathcal{E}$. 
\end{lemma}
\begin{proof}
Our argument uses the theory of bounded holomorphic semigroups, c.f. \cite{Davies1980}. For $f\in L^2(\Omega)$, $r>0$ and $|\theta|\leq \pi/3$ put
\begin{equation*}
f_r=\exp[-re^{i\theta}H_{\lambda,\phi}]f.
\end{equation*}
It follows that $f_r\in\Dom(H_{\lambda,\phi})$ and
\begin{eqnarray*}
\frac{d}{dr}\|f_r\|_2^2&=&-e^{i\theta}(H_{\lambda,\phi}f_r,f_r)-e^{-i\theta}(f_r,H_{\lambda,\phi}f_r)\\
&=&-e^{i\theta}Q_{\lambda,\phi}(f_r)-e^{-i\theta}\overline{Q_{\lambda,\phi}(f_r)}\\
&=&-(e^{i\theta}+e^{-i\theta})Q(f_r)+D_r
\end{eqnarray*}
where
\begin{equation*}
D_r=-e^{i\theta}[Q_{\lambda,\phi}(f_r)-Q(f_r)]-e^{-i\theta}[\overline{Q_{\lambda,\phi}(f_r)}-Q(f_r)].
\end{equation*}
By Hypothesis \ref{hyp:FormCompare},
\begin{equation*}
|D_r|\leq (Q(f_r)+M(1+R(\lambda))\|f\|_2^2)/2
\end{equation*}
and so with the observation that $e^{i\theta}+e^{-i\theta}\geq 1$ for all $|\theta|\leq \pi/3$,
\begin{equation*}
\frac{d}{dr}\|f_r\|_2^2\leq \frac{M}{2}(1+R(\lambda))\|f\|_2^2.
\end{equation*}
Hence,
\begin{equation*}
\|f_r\|_2\leq \exp(M(1+R(\lambda))r/4)\|f\|_2
\end{equation*}
in view of Gr\"{o}nwall's lemma. From the above estimate we have
\begin{eqnarray*}
\lefteqn{\hspace{-1cm}\|\exp[-zH_{\lambda,\phi}-M(1+R(\lambda))z]\|_{2\rightarrow 2}}\\
&\leq&\exp(M(1+R(\lambda))r/4)\exp(-M(1+R(\lambda))\Re(z)/2)\leq 1
\end{eqnarray*}
for all $z=re^{i\theta}$ for $r>0$ and $|\theta|\leq \pi/3$ because $2\Re(z)\geq r$. Theorem 8.4.6 of \cite{Davies1980} yields
\begin{equation*}
\|(H_{\lambda,\phi}+M(1+R(\lambda))/2)\exp[-tH_{\lambda,\phi}-M(1+R(\lambda))t/2]\|_{2\rightarrow 2}\leq \frac{C'}{t}
\end{equation*}
for all $t>0$. It now follows that
\begin{equation*}
\|H_{\lambda,\phi}T_t^{\lambda,\phi}\|_{2\rightarrow 2}\leq\frac{C}{t}\exp(M(1+R(\lambda))t/2)
\end{equation*}
for all $t>0$ where we have put $C=C'+2$.
\end{proof}

\begin{lemma}\label{Hkappalemma}
For any $k\in\mathbb{N}$, there is $C> 0$ such that
\begin{equation*}
\|H_{\lambda,\phi}^ke^{-tH_{\lambda,\phi}}\|_{2\rightarrow 2}\leq \frac{C}{t^k}\exp(M(1+R(\lambda))t/2)
\end{equation*}
for all $t>0$, $\phi\in\mathcal{E}$ and $\lambda\in \mathbb{V}^*$.
\end{lemma}
\begin{proof}
As $-H_{\lambda,\phi}$ is the generator of the semigroup $e^{-tH_{\lambda,\phi}}$, for any $t>0$ and $f\in L^2(\Omega)$, $e^{-tH_{\lambda,\phi}}f\in \Dom(H_{\lambda,\phi}^k)$. We have
\begin{equation*}
H_{\lambda,\phi}^ke^{-tH_{\lambda,\phi}}=\left(H_{\lambda,\phi}e^{-(t/k)H_{\lambda,\phi}}\right)^k
\end{equation*}
and so by the previous lemma
\begin{equation*}
\|H_{\lambda,\phi}^ke^{-tH_{\lambda,\phi}}\|_{2\rightarrow 2}\leq \left(\frac{C}{t}\exp(M(1+R(\lambda)t/2k)\right)^k
\end{equation*}
from which the result follows.
\end{proof}

\section{Off-diagonal estimates}

\noindent In this section, we prove that the semigroup $T_t=e^{-tH}$ has an integral kernel $K_H$ and we deduce off-diagonal estimates for $K_H$. Here we shall assume the notation of the last section and, like before, all statements are to include Hypotheses \ref{hyp:Garding} and \ref{hyp:FormCompare} without explicit mention. 

\begin{lemma}\label{offdiagonalmachinelem}
If the twisted semigroup $T^{\lambda,\phi}_t$ satisfies the ultracontractive estimate
 \begin{equation}\label{twistedultracontraceq}
 \|T^{\lambda,\phi}_t\|_{2\rightarrow \infty}\leq \frac{C}{t^{\mu_{\Lambda}/2}}\exp[M(R(\lambda)+1)t/2]
 \end{equation}
for all $\lambda\in \mathbb{V}^*$, $\phi\in \mathcal{E}$ and $t>0$ where $C,M>0$, then $T_t$ has integral kernel $K_H(t,x,y)=K_H(t,x,\cdot)\in L^1(\Omega)$ satisfying the off-diagonal bound
\begin{equation*}
|K_H(t,x,y)|\leq \frac{C}{t^{\mu_{\Lambda}}}\exp\left(-tMR^{\#}\left(\frac{x-y}{t}\right)+Mt\right)
\end{equation*}
for all $x,y\in\mathbb{V}$ and $t>0$ where $R^{\#}$ is the Legendre-Fenchel transform of $R$ and $M$ and $C$ are positive constants.
\end{lemma}

\begin{proof}
It is clear that the adjoint of $T_t^{\lambda,\phi}$ is $T_t^{-\lambda,\phi}$ and so by duality and \eqref{twistedultracontraceq},
\begin{equation*}
 \|T^{\lambda,\phi}_t\|_{1\rightarrow 2}\leq \frac{C}{t^{\mu_{\Lambda}/2}}\exp[M(R(\lambda)+1)t/2]
\end{equation*}
for $t>0$ where we have replaced $MR(-\lambda)$ by $MR(\lambda)$ in view of Proposition \ref{prop:ComparePoly}. Thus for all $t>0$, $\lambda\in\mathbb{V}^*$ and $\phi\in\mathcal{E}$,
\begin{eqnarray*}
\|T^{\lambda,\phi}_t\|_{1\rightarrow \infty}&\leq&\|T^{\lambda,\phi}_t\|_{1\rightarrow 2}\|T^{\lambda,\phi}_t\|_{2\rightarrow \infty}\\
&\leq&\frac{C}{t^{\mu_{\Lambda}/2}}\exp[M(R(\lambda)+1)t/2]\frac{C}{t^{\mu_{\Lambda}/2}}\exp[M(R(\lambda)+1)t/2]\\
&\leq&\frac{C}{t^{\mu_{\Lambda}}}\exp[Mt(R(\lambda)+1)].
\end{eqnarray*}
The above estimate guarantees that $T_t^{\lambda,\phi}$ has integral kernel $K_H^{\lambda,\phi}(t,x,y)$ satisfying the same bound (see Theorem 2.27 of \cite{Davies1980}). By construction, we also have 
\begin{equation*}
K_H^{\lambda,\phi}(t,x,y)=e^{-\lambda(\phi(x))}K_H(t,x,y)e^{\lambda(\phi(y))} 
\end{equation*}
 where $K_H=K_H^{0,\phi}$ is the integral kernel of $T_t=T_t^{0,\phi}$. Therefore
\begin{equation*}
|e^{-\lambda(\phi(x))}K_H(t,x,y)e^{\lambda(\phi(y))}|\leq\frac{C}{t^{\mu_{\Lambda}}}\exp(Mt(R(\lambda)+1))
\end{equation*}
or equivalently
\begin{equation*}
|K_H(t,x,y)|\leq\frac{C}{t^{\mu_{\Lambda}}}\exp\left(\lambda(\phi(y)-\phi(x))+Mt(R(\lambda)+1)\right)
\end{equation*}
for all $t>0$, $x,y\in\Omega$, $\lambda\in\mathbb{V}^*$ and $\phi\in\mathcal{E}$. In view of Hypothesis \ref{hyp:FormCompare}, for any $x$ and $y\in\Omega$ there is $\phi\in\mathcal{E}$ for which $\phi(x)=x$ and $\phi(y)=y$. Consequently, we have that for all $x,y\in\Omega$, $\lambda\in\mathbb{V}^*$ and $t>0$,
\begin{equation*}
|K_H(t,x,y|\leq\frac{C}{t^{\mu_{\Lambda}}}\exp\left(\lambda(y-x)+Mt(R(\lambda)+1)\right).
\end{equation*}
The proof of the lemma will be complete upon minimizing the above bound with respect to $\lambda\in\mathbb{V}^*$. In this process, we shall see how the Legendre-Fenchel transform appears naturally. For any $x,y\in\Omega$ and $t>0$, we have
\begin{eqnarray*}
 |K_H(t,x,y)|&\leq& \frac{C}{t^{\mu_{\Lambda}}}\inf_{\lambda}\{ \exp\left\{\lambda(y-x)+Mt(R(\lambda)+1)\right\}\}\\
 &\leq&\frac{C}{t^{\mu_{\Lambda}}}\exp\left(-t\sup_{\lambda}\left\{\lambda\left(\frac{x-y}{t}\right)-MR(\lambda)\right\}\right)\exp(Mt)\\
 &\leq&\frac{C}{t^{\mu_{\Lambda}}}\exp\left(-t(MR)^{\#}\left(\frac{x-y}{t}\right)+Mt\right)\\
  &\leq&\frac{C}{t^{\mu_{\Lambda}}}\exp\left(-tMR^{\#}\left(\frac{x-y}{t}\right)+Mt\right)
\end{eqnarray*}
where we replaced $(MR)^{\#}$ by $MR^{\#}$ in view of Corollary \ref{cor:MovingConstants}
\end{proof}

\begin{theorem}\label{thm:Main}
Let $Q$ satisfy Hypotheses \ref{hyp:Garding}, \ref{hyp:FormCompare} and \ref{hyp:kappa} with reference operator $\Lambda$ and associated self-adjoint operator $H$. Let $R$ be the symbol of $\Lambda$ and $\mu_{\Lambda}$ be its homogeneous order. Then the semigroup $T_t=e^{-tH}$ has integral kernel $K_H:(0,\infty)\times \Omega\times \Omega\rightarrow \mathbb{C}$ satisfying
\begin{equation}\label{eq:Main}
|K_H(t,x,y)|\leq \frac{C}{t^{\mu_{\Lambda}}}\exp\left(-tMR^{\#}\left(\frac{x-y}{t}\right)+Mt\right)
\end{equation}
for all $x,y\in\Omega$ and $t>0$ where $R^{\#}$ is the Legendre-Fenchel transform of $R$ and $C$ and $M$ are positive constants. 
\end{theorem}

\begin{proof}
Take $\kappa$ as in Hypothesis \ref{hyp:kappa}. We note that for all $f\in \Dom(\Lambda^{\kappa})$,
\begin{equation*}
\|f\|_{\infty}\leq CQ_{\Lambda^{\kappa}}(f)^{\mu_{\Lambda}/2\kappa}\|f\|_2^{1-\mu_{\Lambda}/\kappa}
\end{equation*}
in view of Lemma \ref{nashlikelem}. The application of the lemma is justified because $\Lambda^{\kappa}$ is positive-homogeneous with $\kappa^{-1}\Exp(\Lambda^{\kappa})=\Exp(\Lambda)$ and, as required, $\mu_{\Lambda^{\kappa}}=\mu_{\Lambda}/\kappa<1$. For $f\in L^2(\Omega)$, set $f_t=T_t^{\lambda,\phi}f$. In view of Hypothesis \ref{hyp:kappa} and Lemmas \ref{twistedsgboundlemma} and \ref{twistedgenandsgboundlemma}, we have
\begin{eqnarray*}
\|f_t\|_{\infty}&\leq& Q_{\Lambda^{\kappa}}(f_t)^{\mu_{\Lambda}/2\kappa}\|f_t\|_2^{1-\mu_{\Lambda}/\kappa}\\
&\leq&C\left(|\langle H_{\lambda,\phi}^{\kappa}f_t,f_t\rangle|+(1+R(\lambda))^{\kappa}\|f_t\|_2^2\right)^{\mu_{\Lambda}/2\kappa}\|f_t\|_2^{1-\mu_{\Lambda}/\kappa}\\
&\leq&C\left(\|H_{\lambda,\phi}^{\kappa}f_t\|_2\|f_t\|_2+(1+R(\lambda))^{\kappa}\|f_t\|_2^2\right)^{\mu_{\Lambda}/2\kappa}\|f_t\|_2^{1-\mu_{\Lambda}/\kappa}\\
&\leq&C\left(\frac{\exp(M(1+R(\lambda))t/4)}{t^\kappa}+(1+R(\lambda))^k\right)^{\mu_{\Lambda}/2\kappa}\\
& &\hspace{2cm}\times\exp(M(1+R(\lambda))t/4)\|f\|_2\\
&\leq&\frac{C}{t^{\mu_{\Lambda}/2}}\exp(M(1+R(\lambda))t/2)\|f\|_2
\end{eqnarray*}
for all $\phi\in\mathcal{E}$ and $\lambda\in\mathbb{V}^*$. In view of Lemma \ref{offdiagonalmachinelem}, the theorem is proved.
\end{proof}

\section{Homogeneous Operators}\label{sec:HHomogeneous}

In this short section, we show that the term $Mt$ in the heat kernel estimate of Theorem \ref{thm:Main} can be removed when $H$, a generally variable-coefficient operator, is ``homogeneous" in the sense given by Definition \ref{def:HHomogeneousOperator} below. Our setting is that in which $\Omega=\mathbb{V}$ and we shall assume throughout this section that $\mu_{\Lambda}<1$. Our arguments follow closely to the work of G. Barbatis and E. B. Davies \cite{Barbatis1996}. \\

\noindent Let $Q$ be a sesquilinear form on $L^2(\mathbb{V})$ satisfying Hypotheses \ref{hyp:Garding} and \ref{hyp:FormCompare} with reference operator $\Lambda$ and associated self-adjoint operator $H$. For any $E\in\Exp(\Lambda)$ (which we keep fixed throughout this section), observe that
\begin{equation*}
(U_sf)(x)=s^{\mu_{\Lambda}/2}f(s^Ex)
\end{equation*}
defines a unitary operator $U_s$ on $L^2(\mathbb{V})$ for each $s>0$ with $U_s^{*}=U_{1/s}$. For each $s>0$, set
\begin{equation*}
H_s=s^{-1}U_{s}^{*}H U_{s}.
\end{equation*}
and note that $H_s$ is a self-adjoint operator on $L^2(\mathbb{V})$. It is easily verified that the sesquilinear form $Q^s$ associated to $H_s$ has
\begin{equation*}
Q^s(f,g)=s^{-1} Q(U_s f,U_s g)
\end{equation*}
for all $f,g$ in the common domain $\Dom(Q^s)=\Dom(Q)=\Dom(\Lambda^{1/2})$. As $Q^s$ is produced by rescaling $Q$, it is clear the $Q^s$ will satisfy Hypotheses \ref{hyp:Garding} and \ref{hyp:FormCompare}. Let us isolate the following special situation:
\begin{definition}\label{def:HHomogeneousOperator}
Assuming the notation above, we say that $H$ is homogeneous provided that $Q^s$ satisfies Hypotheses \ref{hyp:Garding} and \ref{hyp:FormCompare} with the same constants as $Q$ for all $s>0$. In other words, $Q_s$ satisfies the estimates \eqref{eq:Garding} and \eqref{eq:FormCompare1} uniformly for $s>0$.
\end{definition}

\noindent We note that a positive-homogeneous operator $\Lambda$ is homogeneous in the above sense, for our defining property of homogeneous constant-coefficient operators can be written equivalently as $\Lambda_s=\Lambda$ for all $s>0$. In the example section below, we will see that when $H$ is a variable-coefficient partial differential operator consisting only of ``principal terms", the replacement of $H_s$ by $H$ amounts to a rescaling of the arguments of $H$'s coefficients.

\begin{theorem}\label{thm:HHomogeneous}
Let $Q$ be a sesquilinear form on $L^2(\mathbb{V})$ satisfying Hypotheses \ref{hyp:Garding} and \ref{hyp:FormCompare} with reference operator $\Lambda$ and associated self-adjoint operator $H$. Let $R$ and $\mu_{\Lambda}$ be the symbol and homogeneous order of $\Lambda$, respectively. Assume further that $\mu_{\Lambda}<1$ and so Hypothesis \ref{hyp:kappa} is automatically satisfied (in view of Proposition \ref{prop:kappa}) and hence the conclusion to Theorem \ref{thm:Main} is valid. If $H$ is homogeneous, then its heat kernel $K_{H}$ satisfies the estimate
\begin{equation*}
|K_H(t,x,y)|\leq \frac{C}{t^{\mu_{\Lambda}}}\exp\left(-tMR^{\#}\left(\frac{x-y}{t}\right)\right)
\end{equation*}
for all $x,y\in\mathbb{V}$ and $t>0$, where $C$ and $M$ are positive constants.
\end{theorem}

\begin{proof}
Using the fact that $U_s$ is unitary for each $s>0$, it follows that
\begin{equation*}
e^{-tH_s}=e^{-ts^{-1}U_{1/s}HU_s}=U_{1/s}e^{-(t/s)H}U_s
\end{equation*}
for $s,t>0$. Consequently, for $f\in L^2(\mathbb{V})$,
\begin{eqnarray*}
\left(e^{-tH_s}f\right)(x)&=&\int_{\mathbb{V}}s^{-\mu_{\Lambda}}K_H(t/s,s^{-E}x,y)s^{\mu_{\Lambda}}f(s^{E}y)\,dy\\
&=&s^{-\mu_{\Lambda}}\int_{\mathbb{V}}K_H(t/s,s^{-E}x,s^{-E}y)f(y)\,dy
\end{eqnarray*}
for $s,t>0$ and almost every $x\in\mathbb{V}$. Thus, $e^{-tH_s}$ has an integral kernel $K_H^s:(0,\infty)\times\mathbb{V}\times\mathbb{V}\rightarrow\mathbb{C}$ satisfying
\begin{equation*}
K_H^s(t,x,y)=s^{-\mu_{\Lambda}}K_H(t/s,s^{-E}x,s^{-E}y)
\end{equation*}
for $x,y\in\mathbb{V}$. Equivalently,
\begin{equation*}
K_H(t,x,y)=s^{\mu_{\Lambda}}K_H^s(st,s^{E}x,s^{E}y)
\end{equation*}
for $t,s>0$ and $x,y\in\mathbb{V}$. We now apply the same sequence of arguments to the self-adjoint operators $H_s$ and the semigroups $e^{-tH_s}$. Under the hypothesis that $H$ is homogeneous, a careful study reveals that each estimate in the sequence of lemmas preceding Theorem \ref{thm:Main} and the estimates in the proof of Theorem \ref{thm:Main} are independent of $s$. From this, we obtain positive constants $C$ and $M$ for which
\begin{equation*}
|K_H^s(t,x,y)|\leq \frac{C}{t^{\mu_{\Lambda}}}\exp\left(-tMR^{\#}\left(\frac{x-y}{t}\right)+Mt\right)
\end{equation*}
for all $t>0$ and $x,y\in\mathbb{V}$ and this holds uniformly for $s>0$. Consequently,
\begin{eqnarray*}
|K_H(t,x,y)|&\leq &s^{\mu_{\Lambda}}\frac{C}{(st)^{\mu_{\Lambda}}}\exp\left(-(st)MR^{\#}\left(\frac{s^{E}(x-y)}{st}\right)+Mst\right)\\
&\leq & \frac{C}{t^{\mu_{\Lambda}}}\exp\left(-tMR^{\#}\left(\frac{x-y}{t}\right)+Mst\right)
\end{eqnarray*}
for all $s,t>0$ and$x,y\in\mathbb{V}$ where we have used the fact that $I-E\in\Exp(R^{\#})$. The desired estimate follows by letting $s\rightarrow 0$.
\end{proof}

\section{Regularity of $K_H$}\label{sec:KernelRegularity}

In this section, we discuss the regularity of the heat kernel $K_H$. Given a non-empty open subset $\Omega$ of $\mathbb{V}$, we assume that $Q$ is a sesquilinear form on $L^2(\Omega)$ which satisfies Hypotheses \ref{hyp:Garding} and \ref{hyp:FormCompare} with reference operator $\Lambda$ and associated self-adjoint operator $H$. Further, we shall assume that $\mu_\Lambda<1$ (and so Hypothesis \ref{hyp:kappa} is satisfied automatically) and it is with this assumption we show $K_H$ is H\"{o}lder continuous.

\begin{lemma}\label{oneoverRintegrablelemma}
Let $\Lambda$ be a self-adjoint positive-homogeneous operator with real symbol $R$ and homogeneous order $\mu_{\Lambda}$. If $\mu_{\Lambda}<1$, then
\begin{equation*}
\int_{\mathbb{V}^*}\frac{1}{(1+R(\xi))^{1-\epsilon}}d\xi<\infty
\end{equation*}
where $\epsilon=(1-\mu_{\Lambda})/2$. In particular, $(1+R)^{-1}\in L^1(\mathbb{V}^*)$.
\end{lemma}
\begin{proof}
For any Borel set $B$, write $m(B)=\int_{B}d\xi$. It suffices to prove that
\begin{equation*}
\sum_{l=0}^{\infty}\frac{m(F_l)}{2^l}<\infty
\end{equation*}
where $F_l:=\{\xi\in\mathbb{V}^*:2^l\leq R(\xi)^{1-\epsilon}\leq 2^{l+1}\}$. To this end, fix $E\in\Exp(R)$ and observe that, for any $l\geq 1$, 
\begin{eqnarray*}
F_l&=&\left\{\xi:2^{l-1}\leq(t^{-1} R(\xi))^{1-\epsilon}\leq 2^l\right\}\\
&=&\left\{\xi:2^{l-1}\leq R(t^{-E}\xi)^{1-\epsilon}\leq 2^l\right\}\\
&=&\{t^E\xi:2^{l-1}\leq R(\xi)^{1-\epsilon}\leq 2^l\}=t^E F_{l-1}
\end{eqnarray*}
where we have set $t=2^{1/(1-\epsilon)}$. Continuing inductively we see that $F_l=t^{lE}F_0$ for all $l\in\mathbb{N}$ and so it follows that
\begin{equation*}
m(F_l)=\int_{t^{lE}F_0}d\xi=\int_{F_0}\det(t^{lE})d\xi=(t^{l\tr E})m(F_0)=t^{l\mu_{\Lambda}}m(F_0).
\end{equation*}
where we have used the fact that $\mu_{\Lambda}=\tr E^*=\tr E$ because $E^*\in\Exp(\Lambda)$. Consequently,
\begin{equation*}
\sum_{l=0}^{\infty}2^{-l}m(F_l)=m(F_0)\sum_{l=0}^{\infty}2^{-l}(t^{l\mu_{\Lambda}})=m(F_0)\sum_{l=0}^{\infty}\left(2^{-1}t^{\mu_{\Lambda}}\right)^{l}<\infty
\end{equation*}
because $2^{-1}t^{\mu_{\Lambda}}=2^{(\mu_{\Lambda}/(1-\epsilon)-1)}<1$.
\end{proof}

\begin{lemma}\label{holdercontlemma}
Let $|\cdot|$ be a norm on $\mathbb{V}$ and suppose that $\mu_{\Lambda}<1$. There exists $C>0$ such that
\begin{equation*}
\int_{\mathbb{V}*}\frac{|e^{i\xi(x)}-e^{i\xi(y)}|^2}{1+R(\xi)}d\xi\leq C|x-y|^{(1-\mu_{\Lambda})}
\end{equation*}
for all $x,y\in\mathbb{V}$.
\end{lemma}
\begin{proof}
Let $\mathbf{m}\in \mathbb{N}_+^d$ and $\mathbf{v}$ be that guaranteed by Proposition \ref{prop:OperatorRepresentation} and set $E=E_{\mathbf{v}}^{2\mathbf{m}}\in \Exp(\Lambda)$. We note that it suffices to prove the desired estimate where $|\cdot|$ is the Euclidean norm associated the coordinate system defined by $\mathbf{v}$.  In view of the preceding lemma, 
\begin{equation*}
\frac{|e^{i\xi(x)}-e^{i\xi(y)}|^2}{(1+R(\xi))}\leq 4(1+R(\xi))^{-1}\in L^1(\mathbb{V}^*)
\end{equation*} for all $x,y\in\mathbb{V}$. Consequently, it suffices to treat only the case in which $0<|x-y|\leq 1$. In this case, set $t=|x-y|^{-1}$ and observe that
\begin{eqnarray*}
\int_{\mathbb{V}^*}\frac{|e^{i\xi(x)}-e^{i\xi(y)}|^2}{(1+R(\xi))}d\xi&=&\int_{t\leq R(\xi)}\frac{|e^{i\xi(x)}-e^{i\xi(y)}|^2}{(1+R(\xi))}d\xi+\int_{t>R(\xi)}\frac{|e^{i\xi(x)}-e^{i\xi(y)}|^2}{(1+R(\xi))}d\xi\\
&\leq&\int_{t\leq R(\xi)}\frac{4}{R(\xi)}d\xi+\int_{t>R(\xi)}|e^{i\xi(x)}-e^{i\xi(y)}|^2d\xi\\
&\leq&\int_{1\leq R(\xi)}\frac{4}{R(t^{E^*}\xi)}t^{\mu_{\Lambda}}d\xi+\int_{1>R(\xi)}|e^{i\xi(t^{E}x)}-e^{i\xi(t^{E}y)}|^2t^{\mu_{\Lambda}}d\xi\\
&\leq&t^{\mu_{\Lambda}-1}\int_{1\leq R(\xi)}\frac{4}{R(\xi)}d\xi+t^{\mu_{\Lambda}}|t^{E}(x-y)|^2\int_{1>R(\xi)}4|\xi|_{*}^2d\xi\\
\end{eqnarray*}
where $|\cdot|_{*}$ is the corresponding dual norm on $\mathbb{V}^*$. Using Lemma \ref{oneoverRintegrablelemma} and the fact that $|\xi|_{*}^2$ is bounded on the bounded set $\{1>R(\xi)\}$, it follows that
\begin{equation*}
\int_{\mathbb{V}^*}\frac{|e^{i\xi(x)}-e^{i\xi(y)}|^2}{(1+R(\xi))}d\xi\leq C\left(t^{\mu_{\Lambda}-1}+t^{\mu_{\Lambda}}|t^{E}(x-y)|^2\right)
\end{equation*}
for some $C>0$. Given that $\max(\Spec(E))\leq 1/2$ in view of Proposition \ref{prop:OperatorRepresentation}, we have $|t^E(x-y)|\leq t^{1/2}|x-y| $ because $t\geq 1$ and $|\cdot|$ is the Euclidean norm associated to $\mathbf{v}$. Consequently,
\begin{equation*}
\int_{\mathbb{V}^*}\frac{|e^{i\xi(x)}-e^{i\xi(y)}|^2}{(1+R(\xi))}d\xi\leq C\left(t^{\mu_{\Lambda}-1}+t^{\mu_{\Lambda}+1}|x-y|^2\right)= 2C|x-y|^{(1-\mu_{\Lambda})}.
\end{equation*}
\end{proof}

\noindent The following lemma is analogous to Lemma 14 of \cite{Davies1995}.

\begin{lemma}\label{phiexistslem}
Let $Q$ satisfy Hypotheses \ref{hyp:Garding} and \ref{hyp:FormCompare} on $L^2(\Omega)$ with associated self-adjoint operator $H$ and reference operator $\Lambda$ and assume that $\mu_{\Lambda}<1$. There exists a uniformly bounded function $\phi:\Omega\rightarrow L^2(\Omega)$ such that for every $f\in L^2(\Omega)$,
 \begin{equation}\label{phiexistseq}
 \{(H+1)^{-1/2}f\}(x)=\langle f,\phi(x)\rangle
 \end{equation}
for almost every $x\in\Omega$. Moreover, $\phi$ is H\"{o}lder continuous of order $\alpha=(1-\mu_{\Lambda})/2$. In particular, $(H+1)^{-1/2}$ is a bounded operator from $L^2(\Omega)$ into $L^{\infty}(\Omega)$ and for each $f\in L^2(\Omega)$, there is a version of $(H+1)^{-1/2}f$ which is bounded and H\"{o}lder continuous of order $\alpha$.
\end{lemma}
\begin{proof}
In view of \eqref{eq:Garding},
\begin{equation*}
\int_{\mathbb{V}^*}(1+R(\xi))|\widehat {g_*}(\xi)|^2d\xi\leq 2\|(1+H)^{1/2}g\|_2^2
\end{equation*}
for all $g\in W_{\mathbf{v},0}^{\mathbf{m},2}(\Omega)$ where $R$ is the symbol of $\Lambda$ and $g_*$ denotes the extension of $g$ to $\mathbb{V}$ defined by \eqref{eq:ExtensionDefinition}. Also by the Cauchy-Schwarz inequality
\begin{equation*}
\int_{\mathbb{V}^*}(1+R(\xi))^{\epsilon/2}|\widehat{g_*}(\xi)|d\xi\leq C\left(\int_{\mathbb{V}^*}(1+R(\xi))|\widehat{g_*}(\xi)|^2d\xi\right)^{1/2}
\end{equation*}
where
\begin{equation*}
C^2=\int_{\mathbb{V}^*}\frac{(1+R(\xi))^{\epsilon}}{(1+R(\xi))}d\xi<\infty
\end{equation*}
in view of Lemma \ref{oneoverRintegrablelemma}. Consequently, for all $g\in W_{\mathbf{v},0}^{\mathbf{m},2}(\Omega)$, $\widehat{ g_*}\in L^1(\mathbb{V}^*)$ and
\begin{equation}\label{phiexistseq1}
 \|g\|_{\infty}=\|g_*\|_{L^\infty(\mathbb{V})}\leq\int_{\mathbb{V}^*}(1+R(\xi))^{\epsilon/2}|\widehat{g_*}(\xi)|d\xi\leq C\|(1+H)^{1/2}g\|_2.
\end{equation}
So $(H+1)^{1/2}$ is an injective self-adjoint operator and therefore has dense range in $L^2(\Omega)$. We can therefore consider $(H+1)^{-1/2}$, which by \eqref{phiexistseq1} is a bounded operator from $L^2(\Omega)$ into $L^{\infty}(\Omega)$.

Let $|\cdot|$ be a norm on $\mathbb{V}$ and for $f\in L^2(\Omega)$ set $g=(H+1)^{-1/2}f$. For almost every $x,y\in \Omega$ we have
\begin{eqnarray}\label{holderconteqg}\nonumber
|g(x)-g(y)|&\leq&\int_{\mathbb{V}^*}|e^{i\xi(x)}-e^{i\xi(y)}||\widehat{g_*}(\xi)|d\xi\\\nonumber
&\leq&\left(\int_{\mathbb{V}^*}(1+R(\xi))|\widehat{ g_*}(\xi)|^2d\xi\right)^{1/2}\left(\int_{\mathbb{V}^*}\frac{|e^{i\xi(x)}-e^{i\xi(y)}|^2}{(1+R(\xi))}d\xi\right)^{1/2}\\
&\leq& c\|f\|_2\left(\int_{\mathbb{V}^*}\frac{|e^{i\xi(x)}-e^{i\xi(y)}|^2}{(1+R(\xi))}d\xi\right)^{1/2}\leq C\|f\|_2|x-y|^{\alpha}
\end{eqnarray}
in view of the previous lemma. It follows from \eqref{phiexistseq1} that for almost every $x\in\Omega$, there exists $\phi(x)\in L^2(\Omega)$ such that
\begin{equation*}
(H+1)^{-1/2}f(x)=\langle f,\phi(x)\rangle.
\end{equation*}
By putting $f=\phi(x)$, another application of \eqref{phiexistseq1} shows that  $\|\phi(x)\|_2\leq C$. Moreover, \eqref{holderconteqg} guarantees that
\begin{equation*}
|(f,\phi(x)-\phi(y))|\leq C\|f\|_2|x-y|^{\alpha}
\end{equation*}
from which it follows that $\|\phi(x)-\phi(y)\|_2\leq C|x-y|^{\alpha}$ almost everywhere. Finally, redefine $\phi$, so that all of the above statements hold on all of $\Omega$.
\end{proof}

\noindent Our final result of this section shows that the heat kernel $K_H$ can be analytically continued in its time variable to the open half-plane $\mathbb{C}_+$ provided $\mu_{\Lambda}<1$.

\begin{theorem}\label{thm:MainMeasurable}
Let $Q$ satisfy Hypotheses \ref{hyp:Garding} and \ref{hyp:FormCompare} on $L^2(\Omega)$ with associated self-adjoint operator $H$ and reference operator $\Lambda$. Let $R$ be the symbol of $\Lambda$ and $\mu_{\Lambda}$ be its homogeneous order. If $\mu_{\Lambda}<1$, there exists $K_H:\mathbb{C}_+\times\Omega\times\Omega\rightarrow\mathbb{C}$ such that
\begin{equation*}
\left(e^{-zH}f\right)(x)=\int_{\Omega}K_H(z,x,y)f(y)dy
\end{equation*}
for all $f\in L^1(\Omega)\cap L^2(\Omega)$. For fixed $z\in\mathbb{C}_+$, $K_H(z,\cdot,\cdot):\Omega\times\Omega\rightarrow \mathbb{C}$ is H\"{o}lder continuous of order $\alpha=(1-\mu_{\Lambda})/2$. Moreover for each $x,y\in\Omega$, $\mathbb{C}_+\ni z\mapsto K_H(z,x,y)$ is analytic. Finally, there exists constants $C>0$ and $M\geq 0$ such that
\begin{equation*}
|K_H(t,x,y)|\leq \frac{C}{t^{\mu_{\Lambda}}}\exp\left(-tMR^{\#}\left(\frac{x-y}{t}\right)+Mt\right)
\end{equation*}
for all $x,y\in\Omega$ and $t>0$ where $R^{\#}$ is the Legendre-Fenchel transform of $R$ and $C$ and $M$ are positive constants.
\end{theorem}
\begin{proof}
The fact that $e^{-zH}$ is a bounded holomorphic semigroup ensures that $B(z)=(1+H)e^{-zH}$ is a bounded holomorphic function on $L^2(\Omega)$ for $z\in\mathbb{C}_+$. For $x,y\in\Omega$, $z\in \mathbb{C}_+$ define
\begin{equation*}
K(z,x,y):=\langle B(z)\phi(y),\phi(x)\rangle 
\end{equation*}
where $\phi$ is that given by the preceding lemma. It follows that $\mathbb{C}_+\ni z\mapsto K(z,x,y)$ is analytic for any $x,y\in\Omega$. Now for fixed $z\in\mathbb{C}_+$, $K(z,\cdot,\cdot)$ is H\"{o}lder continuous of order $\alpha$. To see this, let $|\cdot|$ be a norm on $\mathbb{V}$ and, with the help of Lemma \ref{phiexistslem}, observe that for $z\in \mathbb{C}_+$,
\begin{eqnarray*}
|K(z,x,y)-K(z,x',y')|&\leq&|K(z,x,y)-K(z,x',y)|+|K(z,x',y)-K(z,x',y')|\\
&\leq& C\|B(z)\|_{2\to 2}\left(\|\phi(x)-\phi(x')\|_2+\|\phi(y)-\phi(y')\|_2\right)\\
&\leq&C \|B(z)\|_{2\to 2}\left (|x-x'|^{2(\alpha/2)}+|y-y'|^{2(\alpha/2)}\right)\\
&\leq&C\|B(z)\|_{2\to 2}\left(|x-x'|^2+|y-y'|^2\right)^{\alpha/2}
\end{eqnarray*}
for all $(x,y),(x',y')\in\Omega\times\Omega$ as claimed. 

It remains to show that $K(z,x,y)$ is the integral kernel of $e^{-zH}$, for then $K_H(t,\cdot,\cdot)=K(t,\cdot,\cdot)$ for $t>0$ and so the final estimate follows from Theorem \ref{thm:Main} in view of Proposition \ref{prop:kappa}. To this end, an appeal to Lemma \ref{phiexistslem} shows that $(H+1)^{-1/2}:L^2(\Omega)\rightarrow L^{\infty}(\Omega)$ is bounded and so $(H+1)^{-1/2}:L^1(\Omega)\rightarrow L^2(\Omega)$ is also bounded by duality. More is true: Using the self-adjointness of $H$ one can check that
\begin{equation*}
\phi_x(y)=\overline{\phi_y(x)}
\end{equation*}
for almost every $x,y\in\Omega$. Here, the variable of integration is that which appears in the subscript. So, for $f\in L^1(\Omega)\cap L^2(\Omega)$,
\begin{eqnarray*}
\left(e^{-Hz}f\right)(x)&=&((H+1)^{-1/2}B(z)(H+1)^{-1/2}f)(x)\\
&=&\int_{\Omega}(B(z)(H+1)^{-1/2}f)(w)\overline{\phi_w(x)}dw\\
&=&\int_{\Omega}\langle f,\phi(w)\rangle\overline{(B(z)\phi(x)}(w)dw\\
&=&\int_{\Omega}\int_{\Omega}f(y)\overline{\phi_y(w)}\overline{(B(z)\phi(x)}(w)dwdy\\
&=&\int_{\Omega}\int_{\Omega}f(y)\phi_w(y)\overline{(B(z)\phi(x)}(w)dwdy\\
&=&\int_{\Omega}\int_{\Omega}(B(z)\phi(y))(w)\overline{\phi_w(x)}dwf(y)dy
\end{eqnarray*}
as desired.

\end{proof}

\section{Super-semi-elliptic operators}

In this section, we consider a class of partial differential operators to which we apply the theory of the preceding sections. We call this class of operators super-semi-elliptic operators, a term motivated by the super-elliptic operators of E. B. Davies \cite{Davies1995} (see also \cite{Barbatis1996,terElst1997}). Naturally, the class of super-semi-elliptic operators defined below includes the class of super-elliptic operators and our results recapture those of \cite{Davies1995}. \\

\noindent Let $\mathbf{m}=(m_1,m_2,\dots,m_d)\in\mathbb{N}_+^d$, $\mathbf{v}=\{v_1,v_2,\dots,v_d\}$ be a basis of $\mathbb{V}$ and take $E=E_{\mathbf{v}}^{2\mathbf{m}}\in\Gl(\mathbb{V})$ in the notation of \eqref{eq:DefofE}. Given a non-empty open subset $\Omega$ of $\mathbb{V}$, consider the sesquilinear form on $L^2(\Omega)$ given by
\begin{equation*}
Q(f,g)=\sum_{\substack{|\alpha:\mathbf{m}|\leq 1\\|\beta:\mathbf{m}|\leq 1}}\int_{\Omega}a_{\alpha,\beta}(x)D_{\mathbf{v}}^\alpha f(x)\overline{D_{\mathbf{v}}^\beta g(x)}\,dx
\end{equation*}
and defined initially for $f,g\in C_0^\infty(\Omega)$. We shall (minimally) require the following conditions for the functions $a_{\alpha,\beta}$:
\begin{enumerate}[label=(C.\arabic*)]
\item\label{cond:meas} The collection
\begin{equation*}
\{a_{\alpha,\beta}(\cdot)\}_{\substack{|\alpha:\mathbf{m}|\leq 1\\|\beta:\mathbf{m}|\leq 1}}\subseteq L^{\infty}(\Omega)
\end{equation*}
and we shall put
\begin{equation*}
\Gamma=\max_{\substack{|\alpha:\mathbf{m}|\leq 1\\|\beta:\mathbf{m}|\leq 1}}\|a_{\alpha,\beta}\|_{\infty}.
\end{equation*}
\item\label{cond:hermitian} For each $x\in\Omega $, the matrix
\begin{equation*}
\left\{a_{\alpha,\beta}(x)\right\}_{\substack{|\alpha:\mathbf{m}|\leq 1\\|\beta:\mathbf{m}|\leq 1}}
\end{equation*}
is Hermitian.
\item\label{cond:compare} There exists $\{A_{\alpha,\beta}:|\alpha:\mathbf{m}|=1,|\beta:\mathbf{m}|=1\}\subseteq \mathbb{R}$ such that
\begin{equation*}
\Lambda:=\sum_{\substack{|\alpha:\mathbf{m}|= 1\\|\beta:\mathbf{m}|= 1}}A_{\alpha,\beta}D_{\mathbf{v}}^{\alpha+\beta}
\end{equation*}
has positive definite symbol $R$ (and so is a positive-homogeneous operator with $E\in \Exp(\Lambda)$ and $\mu_{\Lambda}=|\mathbf{1}:2\mathbf{m}|$) and for some $C\geq 1$,
\begin{equation*}
\frac{3}{4}\sum_{\substack{|\alpha:\mathbf{m}|= 1\\|\beta:\mathbf{m}|= 1}}A_{\alpha,\beta}\eta_{\alpha}\overline{\eta}_\beta\leq \sum_{\substack{|\alpha:\mathbf{m}|= 1\\|\beta:\mathbf{m}|= 1}}a_{\alpha,\beta}(x)\eta_{\alpha}\overline{\eta}_{\beta}\leq C \sum_{\substack{|\alpha:\mathbf{m}|= 1\\|\beta:\mathbf{m}|= 1}}A_{\alpha,\beta}\eta_{\alpha}\overline{\eta}_\beta
\end{equation*}
for all $\eta\in \oplus_{|\alpha:\mathbf{m}|=1}\mathbb{C}$ and almost every $x\in\Omega$.
\end{enumerate}
Under the above conditions, we shall prove that the sesquilinear form $Q$ is symmetric, bounded below and therefore closable. Its closure is then associated to a self-adjoint operator $H$ on $L^2(\Omega)$ formally given by
\begin{equation}\label{eq:Hindivergenceform}
H=\sum_{\substack{|\alpha:\mathbf{m}|\leq 1\\|\beta:\mathbf{m}|\leq 1}}D_{\mathbf{v}}^\beta\left\{a_{\alpha,\beta}(x)D_\mathbf{v}^\alpha\right\}.
\end{equation}
When Conditions \ref{cond:meas}, \ref{cond:hermitian} and \ref{cond:compare} are satisfied, the sesquilinear form $Q$ is said to be \emph{$\{2\mathbf{m},\mathbf{v}\}$-super-semi-elliptic} or simply \emph{super-semi-elliptic}. Correspondingly, we say that the associated self-adjoint operator $H$ is \emph{$\{2\mathbf{m},\mathbf{v}\}$-super-semi-elliptic} or simply \emph{super-semi-elliptic}. For such a sesquilinear form $Q$, we call $\Lambda$ its associated semi-elliptic reference operator and 
\begin{equation*}
\mu_{\Lambda}=\tr E=|\mathbf{1}:2\mathbf{m}|
\end{equation*}
its homogeneous order. As the following proposition shows, there is a constant $C\geq 0$ for which the sesquilinear form $Q+C$, defined by
\begin{equation*}
(Q+C)(f,g)=Q(f,g)+C\langle f,g\rangle
\end{equation*}
for $f,g\in \Dom(Q)$, satisfies Hypothesis \ref{hyp:Garding} with positive-homogeneous reference operator $\Lambda$.

\begin{proposition}\label{prop:SuperSatisfiesHypothesis1}
Let $Q$ be a $\{2\mathbf{m},\mathbf{v}\}$-super-semi-elliptic form on $L^2(\Omega)$. Then $Q$ extends to a closed and symmetric sesquilinear form on $L^2(\Omega)$ (also denoted by $Q$) with domain
\begin{equation*}
\Dom(Q)=\Dom(Q_{\Lambda})=W_{\mathbf{v},0}^{2,\mathbf{m}}(\Omega).
\end{equation*}
Further, $Q$ is bounded below by some constant $-C$ for $C\geq 0$ and the form $Q+C$ satisfies Hypothesis \ref{hyp:Garding} with reference operator $\Lambda$. We denote by $H$ the self-adjoint operator associated to $Q$ (and corresponding formally with \eqref{eq:Hindivergenceform}). If $H$ (and $Q$) consists only of principal terms, i.e.,
\begin{equation}\label{eq:HDivergenceFormHomogeneous}
H=\sum_{\substack{|\alpha:\mathbf{m}|=1\\ |\beta:\mathbf{m}|=1}}D_{\mathbf{v}}^{\beta}\left\{a_{\alpha,\beta}(x)D_{\mathbf{v}}^{\alpha}\right\},
\end{equation}
then $C$ can taken to be $0$ and so $Q$ satisfies Hypotheses \ref{hyp:Garding} with reference operator $\Lambda$.
\end{proposition}

\begin{proof}
For $f\in C_0^{\infty}(\Omega)$, observe that
\begin{multline*}
\frac{3}{4}Q_{\Lambda}(f)+\sum_{|\alpha+\beta:\mathbf{m}|<2}\int_{\Omega}a_{\alpha,\beta}D_{\mathbf{v}}^{\alpha}f\overline{D_{\mathbf{v}}^{\beta}f}dx\\
=\frac{3}{4}\sum_{\substack{|\alpha:\mathbf{m}|= 1\\|\beta:\mathbf{m}|= 1}}\int_{\Omega}A_{\alpha,\beta}D_{\mathbf{v}}^{\alpha}f\overline{D_{\mathbf{v}}^{\beta}f}dx+\sum_{|\alpha+\beta:\mathbf{m}|<2}\int_{\Omega}a_{\alpha,\beta}(x)D_{\mathbf{v}}^{\alpha}f\overline{D_{\mathbf{v}}^{\beta}f}dx\\
\leq \sum_{\substack{|\alpha:\mathbf{m}|\leq 1\\|\beta:\mathbf{m}|\leq 1}}\int_{\Omega}a_{\alpha,\beta}D_{\mathbf{v}}^{\alpha}f\overline{D_{\mathbf{v}}^{\beta}f}dx=Q(f)\\
\leq C\sum_{\substack{|\alpha:\mathbf{m}|= 1\\|\beta:\mathbf{m}|= 1}}\int_{\Omega}A_{\alpha,\beta}D_{\mathbf{v}}^{\alpha}f\overline{D_{\mathbf{v}}^{\beta}f}dx+\sum_{|\alpha+\beta:\mathbf{m}|<2}\int_{\Omega}a_{\alpha,\beta}D_{\mathbf{v}}^{\alpha}f\overline{D_{\mathbf{v}}^{\beta}f}dx\\
\leq C Q_{\Lambda}(f)+\sum_{|\alpha+\beta:\mathbf{m}|<2}\int_{\Omega}a_{\alpha,\beta}D_{\mathbf{v}}^{\alpha}f\overline{D_{\mathbf{v}}^{\beta}f}dx.
\end{multline*}
Thus
\begin{equation}\label{Hsatassump1propeq1}
\frac{3}{4}Q_{\Lambda}(f)+L(f)\leq Q(f)\leq C Q_{\Lambda}(f)+L(f)
\end{equation}
where we have put 
\begin{equation*}
L(f)=\sum_{|\alpha+\beta:\mathbf{m}|<2}\int_{\Omega}a_{\alpha,\beta}D_{\mathbf{v}}^{\alpha}f\overline{D_{\mathbf{v}}^{\beta}f}dx.
\end{equation*}
Using uniform bound on the coefficients $a_{\alpha,\beta}$ and Cauchy-Schwarz inequality we see that
\begin{equation*}
|L(f)|\leq C\sum_{|\alpha+\beta:\mathbf{m}|<2}\int_{\Omega}|D_{\mathbf{v}}^{\alpha}f||D_{\mathbf{v}}^{\beta}f|dx\\\leq C\sum_{|\alpha+\beta:\mathbf{m}|<2}\|D_{\mathbf{v}}^{\alpha}f\|_2\|D_{\mathbf{v}}^{\beta}f\|_2
\end{equation*}
for some $C>0$. For each multi-index $\gamma$ such that $|\gamma:\mathbf{m}|<1$, it follows from Item \ref{item:Scaling1} of Lemma \ref{lem:Scaling} that
\begin{multline*}
\|D_{\mathbf{v}}^{\gamma}f\|_2^2=\int_{\mathbb{V}^*}|\xi^{2\gamma}||\widehat{f_*}(\xi)|^2d\xi\leq \int_{\mathbb{V}^*}(\epsilon R(\xi)+M_{\epsilon})|\widehat{f_*}(\xi)|^2d\xi=\epsilon Q_{\Lambda}(f)+M_{\epsilon}\|f\|_2^2
\end{multline*}
where $\epsilon$ can be taken arbitrarily small. Taking into account all possible multi-indices appearing in $L$, we can produce a positive constant $M$ for which
\begin{equation}\label{Hsatassump1propeq2}
|L(f)|\leq \frac{1}{4} Q_{\Lambda}(f)+M\|f\|_2^2.
\end{equation}
By combining \eqref{Hsatassump1propeq1} and \eqref{Hsatassump1propeq2}, we obtain
\begin{eqnarray*}
\frac{1}{2}Q_{\Lambda}(f)&=&\frac{3}{4}Q_{\Lambda}(f)-\frac{1}{4}Q_{\Lambda}(f)\\
&\leq & Q(f)-L(f)-\frac{1}{4}Q_{\Lambda}(f)\\
&\leq & Q(f)+C\|f\|_2^2\\
&\leq& C_1Q_{\Lambda}(f)+C_2\|f\|_2^2
\end{eqnarray*}
from which the first assertion follows immediately. In the case that $H$ consists only of its principal terms, $L$ is identically $0$ and so the remaining assertion follows from \eqref{Hsatassump1propeq1} at once.
\end{proof}

\noindent To address Hypothesis \ref{hyp:FormCompare} we need to first introduce an appropriate class $\mathcal{E}$. For any integer $l\geq \max 2\kappa\mathbf{m}=\max\{2\kappa m_j:j=1,2,\dots,d\}$, put
\begin{equation*}
\mathcal{F}_l=\left\{\psi\in C_0^{\infty}(\mathbb{R}): \sup_{x\in\mathbb{R}}\left|\frac{d^j\psi}{dx^j}(x)\right|\leq 1\mbox{ for all }j=1,2,\dots,l\right\}
\end{equation*}
where $\kappa$ is that which appears in Hypothesis \ref{hyp:kappa}. We will take $\mathcal{E}$ to be the set of $\phi\in C_{\infty}^{\infty}(\mathbb{V},\mathbb{V})$ for which there are $\psi_1,\psi_2,\dots,\psi_d\in\mathcal{F}_l$ such that
\begin{equation}\label{definingphieq}
(\theta_{\mathbf{v}}\circ\phi\circ\theta_{\mathbf{v}}^{-1})(x_1,x_2,\dots,x_d)=(\psi_1(x_1),\psi_2(x_2),\dots,\psi_d(x_d))
\end{equation}
for all $(x_1,x_2,\dots,x_d)\in\mathbb{R}^d$.

\begin{remark}
What is important for us is that the $j^{th}$-coordinate function of $\theta_{\mathbf{v}}\circ\phi\circ\theta_{\mathbf{v}}^{-1}$ only depends on $x_j$ for each $j=1,2,\dots,d$. 
\end{remark}

\begin{remark}
The requirement that $l\geq \max 2\kappa\mathbf{m}$ is enough to ensure that Hypothesis \ref{hyp:FormCompare} (and later Hypothesis \ref{hyp:kappa}) holds uniformly for $\phi\in\mathcal{E}$. This, essentially, relies on the uniform boundedness of the derivatives of $\phi$ to sufficiently high order. In all statements to follow, we will assume without explicit mention that $l$ is sufficiently large to handle all derivatives under consideration.
\end{remark}

\begin{lemma}\label{lem:twistedderivative}
For each multi-index $\alpha>0$, there exists $C_{\alpha}>0$ such that for all $f\in \Dom(Q)$, $\phi\in\mathcal{E}$ and $\lambda\in\mathbb{V}^*$,
\begin{equation}
|e^{-\lambda(\phi(x))}D_{\mathbf{v}}^{\alpha}(e^{\lambda(\phi)}f)(x)-D_{\mathbf{v}}^{\alpha}f(x)|\leq C_{\alpha}\sum_{0<\beta\leq\alpha}\sum_{0<\gamma\leq\beta}|\lambda^{\gamma}||D_{\mathbf{v}}^{\alpha-\beta}f(x)|
\end{equation}
for almost every $x\in\mathbb{V}$.
\end{lemma}
\begin{proof}
In view of the coordinate charts $(\mathbb{V},\theta_{\mathbf{v}})$ and $(\mathbb{V}^*,\theta_{\mathbf{v}^*})$, we have
\begin{equation*}
\lambda(\phi(x))=(\lambda_1,\lambda_2,\dots,\lambda_d)\cdot(\psi_1(x_1),\psi_2(x_2),\dots,\psi_d(x_d))
\end{equation*}
for $x\in\mathbb{V}$ and $\lambda\in\mathbb{V}^*$ where $\theta_{\mathbf{v}}(x)=(x_1,x_2,\dots,x_d)$ and $\theta_{\mathbf{v}^*}(\lambda)=(\lambda_1,\lambda_2,\dots,\lambda_d)$. So for any multi-index $\beta>0$,
\begin{eqnarray*}
D_{\mathbf{v}}^{\beta}(e^{\lambda(\phi)})&=&\left(i\frac{\partial}{\partial x_1}\right)^{\beta_1}\left(i\frac{\partial}{\partial x_2}\right)^{\beta_2}\cdots\left(i\frac{\partial}{\partial x_d}\right)^{\beta_d}\left(e^{(\lambda_1,\lambda_2,\dots,\lambda_d)\cdot(\psi_1,\psi_2,\dots,\psi_d)}\right)\\
&=&\left(i^{\beta_1}\frac{\partial^{\beta_1}}{\partial x_1^{\beta_1}}e^{\lambda_1\psi_1}\right)\left(i^{\beta_2}\frac{\partial^{\beta_2}}{\partial x_2^{\beta_2}}e^{\lambda_2\psi_2}\right)\cdots\left(i^{\beta_d}\frac{\partial^{\beta_d}}{\partial x_d^{\beta_d}}e^{\lambda_d\psi_d}\right).
\end{eqnarray*}
Using the properties we have required for each $\psi_j$, it follows that
\begin{equation*}
|e^{-\lambda(\phi)}D_{\mathbf{v}}^{\beta}(e^{\lambda(\phi)})|\leq C_{\beta}\prod_{\beta_j\neq 0} \left(\sum_{l=1}^{\beta_j}|\lambda_j^{l}|\right)\leq C_{\beta}\sum_{0<\gamma\leq\beta}|\lambda^{\gamma}|
\end{equation*}
where $C_{\beta}>0$ is independent of $\phi$ and $\lambda$. In view of the Leibniz rule,
\begin{multline*}
\left|e^{-\lambda(\phi(x))}D_{\mathbf{v}}^{\alpha}\left(e^{\lambda(\phi)}f\right)(x)-D_{\mathbf{v}}^{\alpha}f(x)\right|\\
=\left|\sum_{0<\beta\leq\alpha}C_{\alpha,\beta}e^{-\lambda(\phi(x))}D_{\mathbf{v}}^{\beta}\left(e^{\lambda(\phi)}\right)(x)D_{\mathbf{v}}^{\alpha-\beta}f(x)\right|\\
\leq C_{\alpha}\sum_{0<\beta\leq\alpha}\sum_{0<\gamma\leq\beta}|\lambda^{\gamma}||D_{\mathbf{v}}^{\alpha-\beta}f(x)|.
\end{multline*}
for almost every $x\in\mathbb{V}$ where $C_{\alpha}$ is independent of $\lambda$ and $\phi$. The constants $C_{\alpha,\beta}$ appearing in the penultimate line are the standard multi-index combinations.
\end{proof}

\begin{proposition}\label{prop:SuperSatisfiesHypothesis2}
With respect to the class $\mathcal{E}$ above, $Q$ (and so $Q+C$) satisfies Hypothesis \ref{hyp:FormCompare}. 
\end{proposition}
\begin{proof}
Let $x,y\in\mathbb{V}$ and set $(x_1,x_2,\dots,x_d)=\theta_{\mathbf{v}}(x)$ and $(y_1,y_2,\dots,y_d)=\theta_{\mathbf{v}}(y)$. For each pair $x_i,y_i\in\mathbb{R}$ there is $\psi_i\in\mathcal{F}_l$ for which $\psi_i(x_i)=x_i$ and $\psi_i(y_i)=y_i$; such functions can be found by smoothly cutting off the identity while keeping derivatives bounded appropriately. Using this collection of $\psi_i$'s, we define $\phi$ as in \eqref{definingphieq} and note that
\begin{eqnarray*}
\lefteqn{\hspace{-0.75cm}\phi(x)-\phi(y)}\\
\hspace{.75cm}&=&\theta_{\mathbf{v}}^{-1}(\psi_1(x_1),\psi_2(x_2),\dots,\psi_d(x_d))-\theta_{\mathbf{v}}^{-1}(\psi_1(y_1),\psi_2(y_2),\dots,\psi_d(y_d))\\
&=&\theta_{\mathbf{v}}^{-1}(x_1,x_2,\dots,x_d)-\theta_{\mathbf{v}}^{-1}(y_1,y_2,\dots,y_d)\\
&=&x-y
\end{eqnarray*}
as required.

For any $\lambda\in\mathbb{V}^*$, $\phi\in\mathcal{E}$ and $f\in\Dom(Q)$,
\begin{equation*}
Q_{\lambda,\phi}(f)=\sum_{\substack{|\alpha:\mathbf{m}|\leq 1\\|\beta:\mathbf{m}|\leq 1}}\int_{\Omega}a_{\alpha,\beta}(x)D_{\mathbf{v}}^{\alpha}(e^{-\lambda(\phi)}f)(x)\overline{D_{\mathbf{v}}^{\beta}(e^{\lambda(\phi)}f)(x)}dx.
\end{equation*}
Using the uniform boundedness of the collection $\{a_{\alpha,\beta}\}$, we have
\begin{multline*}
|Q_{\lambda,\phi}(f)-Q(f)|\\
=\Big|\sum_{\substack{0<|\alpha:\mathbf{m}|\leq 1\\0<|\beta:\mathbf{m}|\leq 1}}\int_{\Omega}a_{\alpha,\beta}\Big[e^{\lambda(\phi)}D_{\mathbf{v}}^{\alpha}(e^{-\lambda(\phi)}f)\overline{e^{-\lambda(\phi)}D_{\mathbf{v}}^{\beta}(e^{\lambda(\phi)}f)}-D_{\mathbf{v}}^{\alpha}f\overline{D_{\mathbf{v}}^{\beta}f}\Big]dx\Big|\\
=\Big|\sum_{\substack{0<|\alpha:\mathbf{m}|\leq 1\\0<|\beta:\mathbf{m}|\leq 1}}\int_{\Omega}a_{\alpha,\beta}\Big[\left(e^{\lambda(\phi)}D_{\mathbf{v}}^{\alpha}(e^{-\lambda(\phi)}f)-D_{\mathbf{v}}^{\alpha}f\right)\overline{e^{-\lambda(\phi)}D_{\mathbf{v}}^{\beta}(e^{\lambda(\phi)}f)}\\
+D_{\mathbf{v}}^{\alpha}f\left(\overline{e^{-\lambda(\phi)}D_{\mathbf{v}}^{\beta}(e^{\lambda(\phi)}f)-D_{\mathbf{v}}^{\beta}f}\right)\Big]dx\Big|\\
\leq C\sum_{\substack{0<|\alpha:\mathbf{m}|\leq 1\\0<|\beta:\mathbf{m}|\leq 1}}\int_{\Omega}|e^{\lambda(\phi)}D_{\mathbf{v}}^{\alpha}(e^{-\lambda(\phi)}f)-D_{\mathbf{v}}^{\alpha}f||e^{-\lambda(\phi)}D_{\mathbf{v}}^{\beta}(e^{\lambda(\phi)}f)|\\
+|D_{\mathbf{v}}^{\alpha}f||e^{-\lambda(\phi)}D_{\mathbf{v}}^{\beta}(e^{\lambda(\phi)}f)-D_{\mathbf{v}}^{\beta}f|dx\\
\leq C\sum_{\substack{0<|\alpha:\mathbf{m}|\leq 1\\0<|\beta:\mathbf{m}|\leq 1}}\int_{\Omega}|e^{\lambda(\phi)}D_{\mathbf{v}}^{\alpha}(e^{-\lambda(\phi)}f)-D_{\mathbf{v}}^{\alpha}f||e^{-\lambda(\phi)}D_{\mathbf{v}}^{\beta}(e^{\lambda(\phi)}f)-D_{\mathbf{v}}^{\beta}f|\\
+|D_{\mathbf{v}}^{\alpha}f||e^{-\lambda(\phi)}D_{\mathbf{v}}^{\beta}(e^{\lambda(\phi)}f)-D_{\mathbf{v}}^{\beta}f|dx.
\end{multline*}
With the help of Lemma \ref{lem:twistedderivative},
\begin{multline*}
|Q_{\lambda,\phi}(f)-Q(f)|\\
\leq C\sum_{\substack{0<|\alpha:\mathbf{m}|
\leq 1\\0<|\beta:\mathbf{m}| \leq 1}}\sum_{\substack{0<\gamma_{\alpha}\leq\alpha\\0<\gamma_{\beta}\leq\beta}}\sum_{\substack{0<\eta_{\alpha}\leq\gamma_{\alpha}\\0<\eta_{\beta}\leq\gamma_{\beta}}}\int_{\Omega}|\lambda^{\eta_{\alpha}}||D_{\mathbf{v}}^{\alpha-\gamma_{\alpha}}f||\lambda^{\eta_{\beta}}||D_{\mathbf{v}}^{\beta-\gamma_{\beta}}f|dx\\
+C\sum_{\substack{0<|\alpha:\mathbf{m}|\leq 1\\0<|\beta:\mathbf{m}|\leq 1}}\sum_{0<\gamma_{\beta}\leq\beta}\sum_{0<\eta_{\beta}\leq\gamma_{\beta}}\int_{\Omega}|D_{\mathbf{v}}^{\alpha}f||\lambda^{\eta_{\beta}}||D_{\mathbf{v}}^{\beta-\gamma_{\beta}}f|dx\\
\leq C\sum_{\substack{0<|\alpha:\mathbf{m}|\leq 1\\0<|\beta:\mathbf{m}|\leq 1}}\sum_{\substack{0\leq\gamma_{\alpha}\leq\alpha\\0<\gamma_{\beta}\leq\beta}}\sum_{\substack{0\leq\eta_{\alpha}\leq\gamma_{\alpha}\\0<\eta_{\beta}\leq\gamma_{\beta}}}\int_{\Omega}|\lambda^{\eta_{\alpha}}D_{\mathbf{v}}^{\alpha-\gamma_{\alpha}}f||\lambda^{\eta_{\beta}}D_{\mathbf{v}}^{\beta-\gamma_{\beta}}f| dx\\
\end{multline*}
where $C>0$ is independent of $\phi,\lambda$ and $f$. Thus by the Cauchy-Schwarz inequality,
\begin{equation}\label{Hsatassumpprop2eq}
|Q_{\lambda,\phi}(f)-Q(f)|\leq C\sum_{\substack{0<|\alpha:\mathbf{m}|\leq 1\\0<|\beta:\mathbf{m}|\leq 1}}\sum_{\substack{0\leq\gamma_{\alpha}\leq\alpha\\0<\gamma_{\beta}\leq\beta}}\sum_{\substack{0\leq\eta_{\alpha}\leq\gamma_{\alpha}\\0<\eta_{\beta}\leq\gamma_{\beta}}}\|\lambda^{\eta_{\alpha}}D_{\mathbf{v}}^{\alpha-\gamma_{\alpha}}f\|_2\|\lambda^{\eta_{\beta}}D_{\mathbf{v}}^{\beta-\gamma_{\beta}}f\|_2\\.
\end{equation}
It is important to note that for no such summand is $|\beta-{\gamma_{\beta}}:\mathbf{m}|=1$. In view of Lemma \ref{lem:Scaling} and Proposition \ref{prop:SuperSatisfiesHypothesis1} it follows that for all such $\beta$, $\gamma_{\beta}$ and $\eta_{\beta}$,
\begin{eqnarray*}
\|\lambda^{\eta_{\beta}}D_{\mathbf{v}}^{\beta-\gamma_{\beta}}f\|_2^2&=&\int_{\mathbb{V}^*}|\lambda^{2\eta_{\beta}}\xi^{2(\beta-\gamma_{\beta})}||\widehat{f_*}(\xi)|^2d\xi\\
&\leq& \epsilon \int_{\mathbb{V}^*}R(\xi)|\widehat{f_*}(\xi)|^2d\xi+M_{\epsilon}(1+R(\lambda))\|f\|_2^2\\
&\leq& \epsilon Q_{\Lambda}(f)+M_{\epsilon}(1+R(\lambda))\|f\|_2^2\\
&\leq& \epsilon Q(f)+M(1+R(\lambda))\|f\|_2^2
\end{eqnarray*}
where $\epsilon$ can be taken arbitrarily small. For all admissible $\alpha$, $\gamma_{\alpha}$ and $\eta_{\alpha}$, a similar calculation (making use of Lemma \ref{lem:Scaling} and Proposition \ref{prop:SuperSatisfiesHypothesis1}) shows that
\begin{equation*}
\|\lambda^{\eta_{\alpha}}D_{\mathbf{v}}^{\alpha-\gamma_{\alpha}}f\|_2^2\leq M(Q(f)+(1+R(\lambda))\|f\|_2^2)
\end{equation*}
for some $M>0$. Thus for any $\epsilon>0$, each summand in \eqref{Hsatassumpprop2eq} satisfies
\begin{eqnarray*}
\lefteqn{\|\lambda^{\eta_{\alpha}}D_{\mathbf{v}}^{\alpha-\gamma_{\alpha}}f\|_2\|\lambda^{\eta_{\beta}}D_{\mathbf{v}}^{\beta-\gamma_{\beta}}f\|_2}\\
&\leq& (M(Q(f)+(1+R(\lambda))\|f\|_2^2))^{1/2}(\epsilon Q(f)+M(1+R(\lambda))\|f\|_2^2)^{1/2}\\ 
&\leq& (\epsilon M)^{1/2} Q(f)+\frac{M^{3/2}}{\epsilon^{1/2}}(1+R(\lambda))\|f\|_2^2.
\end{eqnarray*}
The result now follows by choosing $\epsilon$ appropriately and combining these estimates. 
\end{proof}

\subsection{When $\mu_{\Lambda}=|\mathbf{1}:2\mathbf{m}|<1$}

Let $Q$ be a $\{2\mathbf{m},\mathbf{v}\}$-super-semi-elliptic form on $L^2(\Omega)$ with reference operator $\Lambda$ (with symbol $R$ and homogeneous order $\mu_{\Lambda}$) and associated super-semi-elliptic operator $H$. Throughout this subsection we investigate the case in which 
\begin{equation*}
\mu_{\Lambda}=|\mathbf{1}:2\mathbf{m}|=\sum_{j=1}^d \frac{1}{2m_j}<1.
\end{equation*}
In view of Propositions \ref{prop:kappa}, \ref{prop:SuperSatisfiesHypothesis1} and \ref{prop:SuperSatisfiesHypothesis2}, the sesquilinear form $Q+C$ satisfies Hypotheses \ref{hyp:Garding}, \ref{hyp:FormCompare} and \ref{hyp:kappa}. Upon noting that the semigroup generated by $-H$ and that generated by $-(H+C)$ are related by $e^{-t(H+C)}=e^{-tC}e^{-tH}$, the results of Section \ref{sec:KernelRegularity} immediately give us the following proposition.

\begin{proposition}\label{prop:Super}
Let $Q$ be a $\{2\mathbf{m},\mathbf{v}\}$-super-semi-elliptic form on $L^2(\Omega)$ with reference operator $\Lambda$ and associated self-adjoint super-semi-elliptic operator $H$. Let $R$ be the symbol and $\mu_{\Lambda}=|\mathbf{1}:2\mathbf{m}|$ be the homogeneous order of $\Lambda$, respectively. If $\mu_{\Lambda}<1$, then the semigroup $T_z=e^{-zH}$ has integral kernel $K_H:\mathbb{C}_+\times\Omega\times\Omega\rightarrow\mathbb{C}$ for which
\begin{equation*}
\left(e^{-zH}f\right)(x)=\int_\Omega K_H(z,x,y)f(y)\,dy
\end{equation*}
for all $f\in L^1(\Omega)\cap L^2(\Omega)$. For fixed $z$, $K_H(z,\cdot,\cdot)$ is jointly H\"{o}lder continuous of order $\alpha=(1-\mu_{\Lambda})/2$. For fixed $x,y\in\Omega$, $z\mapsto K_H(z,x,y)$ is analytic on $\mathbb{C}_+$. Finally, there are constants $C>0$ and $M\geq 0$ for which
\begin{equation*}
|K_H(t,x,y)|\leq \frac{C}{t^{\mu_{\Lambda}}}exp\left(-tMR^{\#}\left(\frac{x-y}{t}\right)+Mt\right)
\end{equation*}
for all $x,y\in\Omega$ and $t>0$ where $R^{\#}$ is the Legendre-Fenchel transform of $R$.
\end{proposition}

\noindent Let us now focus on the special case in which $\Omega=\mathbb{V}$ and the super-semi-elliptic form $Q$ (and $H$) consist only of principle terms, i.e.,
\begin{equation*}
Q(f,g)=\sum_{\substack{|\alpha:\mathbf{m}|=1\\|\beta:\mathbf{m}|=1}}\int_{\mathbb{V}}a_{\alpha,\beta}(x)D_{\mathbf{v}}^\alpha f(x)\overline{D_{\mathbf{v}}^{\beta}g(x)}\,dx,
\end{equation*}
or, equivalently, $H$ is the form \eqref{eq:HDivergenceFormHomogeneous}. We will continue to assume that $\mu_{\Lambda}=|\mathbf{1}:2\mathbf{m}|<1$. In the notation of Section \ref{sec:HHomogeneous}, we observe that
\begin{eqnarray*}
Q^s(f,g)&=&s^{-1}Q(U_s f,U_s g)\\
&=&s^{-1}\sum_{\substack{|\alpha:\mathbf{m}|=1\\|\beta:\mathbf{m}|=1}}\int_{\mathbb{V}}a_{\alpha,\beta}(x)D_{\mathbf{v}}^{\alpha}(s^{\mu_{\Lambda}/2}f_s)(x)\overline{D_{\mathbf{v}}^{\alpha}(s^{\mu_{\Lambda}/2}g_s)(x)}\,dx\\
&=&s^{-1}s^{\mu_{\Lambda}}\sum_{\substack{|\alpha:\mathbf{m}|=1\\|\beta:\mathbf{m}|=1}}\int_{\mathbb{V}}a_{\alpha,\beta}(x)D_{\mathbf{v}}^{\alpha}(f_s)(x)\overline{D_{\mathbf{v}}^{\alpha}(g_s)(x)}\,dx
\end{eqnarray*} 
for $f,g\in \Dom(Q)$ where by $f_s$ (and $g_s$) is defined by $f_s(x)=f(s^Ex)$ for $s>0$ and $x\in\mathbb{V}$. Noting the definition of $E$ at the the beginning of the section, for each multi-index $\gamma$ such that $|\gamma:\mathbf{m}|=1$,
\begin{equation*}
D_{\mathbf{v}}^{\gamma}f_s(x)=s^{|\gamma:2\mathbf{m}|}(D_{\mathbf{v}}^{\gamma}f)(s^Ex)=s^{1/2}(D_{\mathbf{v}}^{\gamma}f)(s^Ex).
\end{equation*}
Therefore, by a change of variables, we obtain
\begin{eqnarray*}
Q^s(f,g)&=&s^{\mu_{\Lambda}}\sum_{\substack{|\alpha:\mathbf{m}|=1\\|\beta:\mathbf{m}|=1}}\int_{\mathbb{V}}a_{\alpha,\beta}(x)D_{\mathbf{v}}^{\alpha}f(s^Ex)\overline{D_{\mathbf{v}}^{\alpha}f(s^Ex)}\,dx\\
&=&\sum_{\substack{|\alpha:\mathbf{m}|=1\\|\beta:\mathbf{m}|=1}}\int_{\mathbb{V}}a_{\alpha,\beta}(s^{-E}x)D_{\mathbf{v}}^{\alpha}f(x)\overline{D_{\mathbf{v}}^{\alpha}f(x)}\,dx
\end{eqnarray*}
for all $f,g\in \Dom(Q)$. Under our assumption that $Q$ is super-semi-elliptic, all estimates concerning $a_{\alpha,\beta}$ hold uniformly for $x\in\mathbb{V}$, and we may therefore conclude that the associated self-adjoint operator $H$ is homogeneous in the sense of in the sense of Section \ref{sec:HHomogeneous}. Consequently,  an appeal to Theorem \ref{thm:HHomogeneous} guarantees that the heat kernel $K_H$ satisfies
\begin{equation*}
|K_H(t,x,y)|\leq \frac{C}{t^{\mu_{\Lambda}}}\exp\left(-tMR^{\#}\left(\frac{x-y}{t}\right)\right)
\end{equation*}
for all $t>0$ and $x,y\in \mathbb{V}$ where $C$ and $M$ are positive constants. \\

\subsection{When $\mu_{\Lambda}=|\mathbf{1},2\mathbf{m}|\geq 1$}

\noindent In the last subsection, we deduced heat kernel estimates for $\{2\mathbf{m},\mathbf{v}\}$-super-semi-elliptic operators in the case that $\mu_{\Lambda}=|\mathbf{1}:2\mathbf{m}|<1$; in this setting Hypothesis \ref{hyp:kappa} was met trivially by virtue of Proposition \ref{prop:kappa}. In general, we expect these results to also be valid in the case that $\mu_{\Lambda}=|\mathbf{1}:2\mathbf{m}|=1$ (by the methods of \cite{Auscher1998} and \cite{terElst1997}); however we do not pursue this here. As discussed in the introduction, without additional assumptions on the regularity of the coefficients, these results cannot be pushed into the realm in which $\mu_{\Lambda}=|\mathbf{1}:2\mathbf{m}|>1$. For an account of the relevant counterexamples which pertain to elliptic operators with measureable coefficients, we encourage the reader to see \cite{Davies1997a,deGiorgi1968,Mazya1968}; further discussion can be found in Section 4.1 of \cite{Davies1997}.\\

\noindent We here investigate the situation in which a $\{2\mathbf{m},\mathbf{v}\}$-super-semi-elliptic form $Q$ has $\mu_{\Lambda}=|\mathbf{1}:2\mathbf{m}|$ unrestricted (allowing for $\mu_{\Lambda}\geq 1$). In this situation, it is possible that $\kappa>1$ and so Hypothesis \ref{hyp:kappa} does not, in general, follow from Proposition \ref{prop:kappa}. We must therefore verify the hypothesis directly. In line with the remarks of the previous paragraph, we shall make some additional (strong) assumptions concerning the regularity of the coefficients $\{a_{\alpha,\beta}\}$ under which the verification of Hypothesis \ref{hyp:kappa} is relatively straightforward. 

To this end, let $Q$ be a $\{2\mathbf{m},\mathbf{v}\}$-super-semi-elliptic form on $L^2(\mathbb{V})$ with coefficients $\{a_{\alpha,\beta}\}$. In addition to Conditions \ref{cond:meas}, \ref{cond:hermitian} and \ref{cond:compare}, we ask that the following two conditions are satisfied:
\begin{enumerate}[label=(C.\arabic*)]
\setcounter{enumi}{3}
\item\label{cond:smooth} \begin{equation*}\{a_{\alpha,\beta}(\cdot)\}_{\substack{|\alpha:\mathbf{m}|\leq 1\\ |\beta:\mathbf{m}|\leq 1}}\subseteq C^{\infty}(\mathbb{V})
\end{equation*}
\item\label{cond:constanttopsymbol} For each pair of multi-indices $\alpha$ and $\beta$ for which $|\alpha:\mathbf{m}|=|\beta:\mathbf{m}|=1$, the function $a_{\alpha,\beta}(\cdot)$ is identically constant.
\end{enumerate}

\noindent In view of Conditions \ref{cond:compare} and \ref{cond:constanttopsymbol}, we may assume without loss of generality that the principal part of $Q$ is given by $\Lambda$. In other words, we assume that $a_{\alpha,\beta}=A_{\alpha,\beta}\in\mathbb{R}$ for each  $\alpha$ and $\beta$ for which $|\alpha:\mathbf{m}|=|\beta:\mathbf{m}|=1$. This allows us to write
\begin{equation*}
Q(f,g)=\langle \Lambda f,g\rangle+L(f,g)
\end{equation*}
for all $f,g\in C_0^{\infty}(\mathbb{V})$ where
\begin{equation*}
\Lambda=\sum_{\substack{|\alpha:\mathbf{m}|=1\\|\beta:\mathbf{m}|=1}}A_{\alpha,\beta}D_{\mathbf{v}}^{\alpha+\beta}
\end{equation*}
and
\begin{equation*}
L(f,g)=\sum_{|\alpha+\beta:\mathbf{m}|<2}\int_{\mathbb{V}}a_{\alpha,\beta}(x)D_{\mathbf{v}}^{\alpha}f(x)\overline{D_{\mathbf{v}}^{\beta}g(x)}\,dx.
\end{equation*}
Furthermore, it is easy to see that Condition \ref{cond:smooth} ensures that the formal expression \eqref{eq:Hindivergenceform} makes sense. More precisely, if we define the differential operator $H_0$ by
\begin{equation*}
H_0f(x)=\Lambda f(x)+\sum_{|\alpha+\beta:\mathbf{m}|<2}D_{\mathbf{v}}^{\beta}\left\{a_{\alpha,\beta}(x)D_{\mathbf{v}}^{\alpha} f(x)\right\}
\end{equation*}
for $f\in \Dom(H_0):=C_0^{\infty}(\mathbb{V})$, then $H_0f=Hf$ whenever $f\in C_0^{\infty}(\mathbb{V})$. More is true.
\begin{proposition}\label{prop:ESA}
Assume Conditions \ref{cond:meas}-\ref{cond:constanttopsymbol} hold. For each integer $\kappa\geq 1$, define the linear differential operator $H_0^{\kappa}$ by
\begin{equation*}
H_0^{\kappa}f=(H_0)^{\kappa} f
\end{equation*}
with domain $\Dom(H_0^{\kappa})=C_0^\infty(\mathbb{V})$. Then the following properties hold:
\begin{enumerate}
\item There are smooth functions $b_{\alpha,\beta}=\overline{b_{\beta,\alpha}}$ for $|\alpha+\beta:\mathbf{m}|<2\kappa$ and real constants $B_{\alpha,\beta}=B_{\beta,\alpha}$ for $|\alpha:\mathbf{m}|=|\beta:\mathbf{m}|=\kappa$ for which
\begin{eqnarray}\label{eq:HkappaFormal}\nonumber
H_0^{\kappa}f&=&\Lambda^{\kappa}f+\sum_{|\alpha+\beta:\mathbf{m}|<2\kappa}D_{\mathbf{v}}^{\beta}\left\{b_{\alpha,\beta}D_{\mathbf{v}}^{\alpha}f\right\}\\
&=&\sum_{\substack{|\alpha:\mathbf{m}|=\kappa\\|\beta:\mathbf{m}|=\kappa}}D_{\mathbf{v}}^{\beta}\left\{B_{\alpha,\beta}D_{\mathbf{v}}^{\alpha}f\right\}+\sum_{|\alpha+\beta:\mathbf{m}|<2\kappa}D_{\mathbf{v}}^{\beta}\left\{b_{\alpha,\beta}D_{\mathbf{v}}^{\alpha}f\right\}
\end{eqnarray}
for $f\in C_0^{\infty}(\mathbb{V})$.
\item $H_0^{\kappa}$ with initial domain $Dom(H_0^{\kappa})=C_0^{\infty}(\mathbb{V})$ is essentially self-adjoint, its closure is precisely the self-adjoint operator $H^{\kappa}$ (defined as the $\kappa$th power of $H$) and
\begin{equation*}
\Dom(H^{\kappa})=W_{\mathbf{v}}^{\mathbf{2\kappa \mathbf{m}},2}(\mathbb{V}).
\end{equation*}
\end{enumerate}
\end{proposition}
\begin{proof}
The first statement follows by direct calculation (Leibniz's rule) keeping in mind that $a_{\alpha,\beta}(x)$ are bounded smooth functions forming a Hermitian matrix at each $x\in\mathbb{V}$. Using integration by parts and the definition of $Q$, for $g\in \Dom(H_0^k)$, we find that
\begin{equation*}
\langle H^{\kappa} f,g\rangle=Q(H^{\kappa-1}f,g)=\langle H^{\kappa-1}f,H_0 g\rangle=\cdots =\langle f,H_0^k g\rangle
\end{equation*}
for all $f\in \Dom(H^{\kappa})$. In view of the self-adjointness of $H^{\kappa}$, this calculation guarantees that $g\in \Dom((H^{\kappa})^*)=\Dom(H^{\kappa})$ and $H^{\kappa}g=H_0^{\kappa}g$ and therefore $H^{\kappa}$ is a self-adjoint extension of the symmetric operator $H_0^{\kappa}$. It remains to show that this operator is essentially self adjoint and characterize the domain of $H^{\kappa}$. 

In view of the first statement, we write
\begin{equation*}
H_0^{\kappa}=\Lambda^{\kappa}+\Psi
\end{equation*}
where $\Psi$ is the symmetric operator
\begin{equation*}
\Psi=\sum_{|\alpha+\beta:\kappa\mathbf{m}|<2}D_{\mathbf{v}}^{\beta}\left\{b_{\alpha,\beta}D_{\mathbf{v}}^{\alpha}\right\}
\end{equation*}
with domain $\Dom(\Psi)=\Dom(H_0^{\kappa})$. It is straightforward to see that $\Lambda^{\kappa}$ is a positive-homogeneous operator with symbol $R(\xi)^{\kappa}$.  Further, observe that $E':=E_{\mathbf{v}}^{2\kappa\mathbf{m}}\in \Exp(\Lambda^{\kappa})$ and so the homogeneous order of $\Lambda^{\kappa}$ is $\mu_{\Lambda^{\kappa}}=\tr E'=\mu_{\Lambda}/\kappa$. An appeal to Proposition \ref{prop:esa} guarantees that $\Lambda^\kappa$, with initial domain $C_0^{\infty}(\mathbb{V})$, is essentially self-adjoint and the domain of its closure is $W_{\mathbf{v}}^{2\kappa\mathbf{m},2}(\mathbb{V})$.

By analogous arguments to those given in the proof of Proposition \ref{prop:SuperSatisfiesHypothesis1} using the fact that $|\alpha+\beta:\kappa\mathbf{m}|<2$ for all multi-indices appearing in $\Psi$, by virtue of Lemma \ref{lem:Scaling} we find that for any $\epsilon>0$, there is $M_{\epsilon}\geq 1$ for which
\begin{equation*}
\|\Psi f\|_2\leq \epsilon \|\Lambda^{\kappa} f\|_2+M_{\epsilon}\|f\|_2
\end{equation*}
for all $f\in C_0^{\infty}(\mathbb{V})$. In view of this estimate, an appeal to Lemma 7.4 of \cite{Schechter1986} ensures that $H_0^{\kappa}=\Lambda^{\kappa}+\Psi$ is essentially self-adjoint and its closure has domain $W_{\mathbf{v}}^{2\kappa\mathbf{m},2}(\mathbb{V})$. Upon noting that $H^{\kappa}$ is a self-adjoint extension of $H_0^{\kappa}$, it is therefore the unique self-adjoint extension and we may conclude at once that
\begin{equation*}
\Dom(H^{\kappa})=W_{\mathbf{v}}^{2\kappa\mathbf{m},2}(\mathbb{V}).
\end{equation*}
\end{proof}

\noindent The following lemma contains the essential estimate needed to verify Hypothesis \ref{hyp:kappa} for a super-semi-elliptic operator whose coefficients satisfy Conditions \ref{cond:meas}-\ref{cond:constanttopsymbol}.

\begin{lemma}
Assume Conditions \ref{cond:meas}-\ref{cond:constanttopsymbol} hold and let $\kappa\geq 1$ be an integer. Then, for any $\epsilon>0$, there is a constant $M_{\epsilon}\geq 1$ for which
\begin{equation*}
|\langle H_{\lambda,\phi}^{\kappa}f,f\rangle-Q_{\Lambda^{\kappa}}(f)|\leq \epsilon Q_{\Lambda^{\kappa}}(f)+M_{\epsilon}(1+R(\lambda))^{\kappa}\|f\|_2^2
\end{equation*}
for all $\lambda\in\mathbb{V}^*$, $\phi\in\mathcal{E}$ and $f\in C_0^{\infty}(\mathbb{V})$.
\end{lemma}
\begin{proof}
It follows from the previous proposition that, for $\lambda\in\mathbb{V}^*$, $\phi\in \mathcal{E}$ and $f\in C_0^{\infty}(\mathbb{V})$,
\begin{equation*}
H_{\lambda,\phi}^{\kappa} f=(H_0^{\kappa})_{\lambda,\phi}f=e^{\lambda(\phi)}H_0^{\kappa}(e^{-\lambda(\phi)}f)
\end{equation*}
where $H_0^{\kappa}f$ is given by \eqref{eq:HkappaFormal}. With this in mind, integration by parts gives
\begin{equation*}
\langle H_{\lambda,\phi}^{\kappa}f,f\rangle-Q_{\Lambda^{\kappa}}(f)=U(\lambda,\phi,f)+W(\lambda,\phi,f)
\end{equation*}
where
\begin{eqnarray*}
U(\lambda,\phi,f)&=&\sum_{\substack{|\alpha:\mathbf{m}|=\kappa\\|\beta:\mathbf{m}|=\kappa}}B_{\alpha,\beta}\int_{\mathbb{V}}\Big[\left(e^{\lambda(\phi)}D_{\mathbf{v}}^{\alpha}\left(e^{-\lambda(\phi)}f\right)\right)\overline{\left(e^{-\lambda(\phi)}D_{\mathbf{v}}^{\beta}\left(e^{\lambda(\phi)}f\right)\right)}\\
&&\hspace{7cm}-D_{\mathbf{v}}^\alpha f\overline{D_{\mathbf{v}}^{\beta} f}\,\Big]\,dx
\end{eqnarray*}
and
\begin{equation*}
W(\lambda,\phi,f)=\sum_{|\alpha+\beta:\mathbf{m}|<2\kappa}\int_{\mathbb{V}}b_{\alpha,\beta}\left(e^{\lambda(\phi)}D_{\mathbf{v}}^{\alpha}\left(e^{-\lambda(\phi)}f\right)\right)\overline{\left(e^{-\lambda(\phi)}D_{\mathbf{v}}^{\beta}\left(e^{\lambda(\phi)}f\right)\right)}\,dx
\end{equation*}
for $\lambda\in\mathbb{V}^*$, $\phi\in\mathcal{E}$ and $f\in C_0^{\infty}(\mathbb{V})$. Just as we did in the proof to Proposition \ref{prop:SuperSatisfiesHypothesis2}, we write
\begin{multline*}
|U(\lambda,\phi,f)|\\
=\Big|\sum_{\substack{|\alpha:\mathbf{m}|=\kappa\\|\beta:\mathbf{m}|=\kappa}}B_{\alpha,\beta}\int_{\mathbb{V}}\Big[\left(e^{\lambda(\phi)}D_{\mathbf{v}}^{\alpha}(e^{-\lambda(\phi)}f)-D_{\mathbf{v}}^{\alpha}f\right)\overline{e^{-\lambda(\phi)}D_{\mathbf{v}}^{\beta}(e^{\lambda(\phi)}f)}\\
+D_{\mathbf{v}}^{\alpha}f\left(\overline{e^{-\lambda(\phi)}D_{\mathbf{v}}^{\beta}(e^{\lambda(\phi)}f)-D_{\mathbf{v}}^{\beta}f}\right)\Big]dx\Big|\\
\leq C\sum_{\substack{|\alpha:\mathbf{m}|=\kappa\\|\beta:\mathbf{m}|=\kappa}}\int_{\mathbb{V}}|e^{\lambda(\phi)}D_{\mathbf{v}}^{\alpha}(e^{-\lambda(\phi)}f)-D_{\mathbf{v}}^{\alpha}f||e^{-\lambda(\phi)}D_{\mathbf{v}}^{\beta}(e^{\lambda(\phi)}f)|\\
+|D_{\mathbf{v}}^{\alpha}f||e^{-\lambda(\phi)}D_{\mathbf{v}}^{\beta}(e^{\lambda(\phi)}f)-D_{\mathbf{v}}^{\beta}f|dx\\
\leq C\sum_{\substack{|\alpha:\mathbf{m}|=\kappa\\|\beta:\mathbf{m}|=\kappa}}\int_{\mathbb{V}}|e^{\lambda(\phi)}D_{\mathbf{v}}^{\alpha}(e^{-\lambda(\phi)}f)-D_{\mathbf{v}}^{\alpha}f||e^{-\lambda(\phi)}D_{\mathbf{v}}^{\beta}(e^{\lambda(\phi)}f)-D_{\mathbf{v}}^{\beta}f|\\
+|D_{\mathbf{v}}^{\alpha}f||e^{-\lambda(\phi)}D_{\mathbf{v}}^{\beta}(e^{\lambda(\phi)}f)-D_{\mathbf{v}}^{\beta}f|dx.
\end{multline*}
where $C$ is independent of $\lambda$, $\phi$ and $f$. With the help of Lemma \ref{lem:twistedderivative} and the Cauchy-Schwarz inequality, we have
\begin{multline*}
|U(,\lambda, \phi,f)|\leq C \sum_{\substack{|\alpha:\mathbf{m}|=\kappa\\|\beta:\mathbf{m}|=\kappa}}\sum_{\substack{0<\rho_{\alpha}\leq \alpha\\0<\rho_{\beta}\leq \beta}} \sum_{\substack{0<\gamma_\alpha\leq \rho_{\alpha}\\0<\gamma_{\beta}\leq \rho_{\beta}}}  \int_{\mathbb{V}}|\lambda^{\gamma_{\alpha}}||D_{\mathbf{v}}^{\alpha-\rho_{\alpha}}f||\lambda^{\gamma_{\beta}}||D_{\mathbf{v}}^{\beta-\rho_{\beta}}f|\,dx\\
+C\sum_{\substack{|\alpha:\mathbf{m}|=\kappa\\|\beta:\mathbf{m}|=\kappa}}\sum_{0<\rho_{\alpha}\leq \alpha} \sum_{0<\gamma_{\alpha}\leq \rho_{\alpha}} \int_{\mathbb{V}}|\lambda^{\gamma_{\alpha}}||D_{\mathbf{v}}^{\alpha-\rho_{\alpha}}f||D_{\mathbf{v}}^{\beta}|\,dx\\
\leq C\sum_{\substack{|\alpha:\mathbf{m}|=\kappa\\|\beta:\mathbf{m}|=\kappa}}\sum_{\substack{0<\rho_{\alpha}\leq \alpha\\0\leq \rho_{\beta}\leq \beta}} \sum_{\substack{0<\gamma_\alpha\leq \rho_{\alpha}\\0\leq\gamma_{\beta}\leq \rho_{\beta}}}  \int_{\mathbb{V}}|\lambda^{\gamma_{\alpha}}||D_{\mathbf{v}}^{\alpha-\rho_{\alpha}}f||\lambda^{\gamma_{\beta}}||D_{\mathbf{v}}^{\beta-\rho_{\beta}}f|\,dx\\
\leq C\sum_{\substack{|\alpha:\mathbf{m}|=\kappa\\|\beta:\mathbf{m}|=\kappa}}\sum_{\substack{0<\rho_{\alpha}\leq \alpha\\0\leq \rho_{\beta}\leq \beta}} \sum_{\substack{0<\gamma_\alpha\leq \rho_{\alpha}\\0\leq\gamma_{\beta}\leq \rho_{\beta}}}\|\lambda^{\gamma_{\alpha}}D_{\mathbf{v}}^{\alpha-\rho_{\alpha}}f\|_2\|\lambda^{\gamma_{\beta}}D_{\mathbf{v}}^{\beta-\rho_{\beta}}f\|_2
\end{multline*}
where, again, $C$ is independent of $\lambda$, $\phi$ and $f$. For each $\alpha$, $\rho_{\alpha}$ and $\gamma_{\alpha}$ such that $|\alpha:\mathbf{m}|=\kappa$, $0<\rho_{\alpha}\leq\alpha$ and $0<\gamma_{\alpha}\leq\rho_\alpha$, properties of Fourier transform and Lemma \ref{lem:Scaling} guarantee that, for any $\epsilon>0$ there is $M_{\epsilon}\geq 1$ for which
\begin{multline*}
\|\lambda^{\gamma_{\alpha}}D_{\mathbf{v}}^{\alpha-\rho_{\alpha}}f\|_2^2=\|\lambda^{\gamma_{\alpha}}\xi^{\alpha-\rho_{\alpha}}\hat{f}\|_{2^*}^2=\int_{\mathbb{V}^*}|\lambda^{2\gamma_{\alpha}}\xi^{2\alpha-2\rho_{\alpha}}||\hat{f}(\xi)|^2\,d\xi\\
\leq \int_{\mathbb{V}^*}\Big(\epsilon R(\xi)^{\kappa}+M_\epsilon(R(\lambda)+1)^{\kappa}\Big)|\hat{f}(\xi)|^2\,d\xi\\
\leq \epsilon Q_{\Lambda^{\kappa}}(f)+M_{\epsilon}(R(\lambda)+1)^{\kappa}\|f\|_2^2.
\end{multline*}
Similarly, for each $\beta,\rho_{\beta}$ and $\gamma_{\beta}$ such that $|\beta:\mathbf{m}|=\kappa$, $0\leq \rho_{\beta}\leq\beta$ and $0\leq \gamma_\beta\leq \rho_{\beta}$, there is a constant $M$ for which
\begin{equation*}
\|\lambda^{\gamma_{\alpha}}D_{\mathbf{v}}^{\alpha-\rho_{\alpha}}f\|_2^2\leq M\left(Q_{\Lambda^{\kappa}}(f)+(1+R(\lambda))^{\kappa}\|f\|_2^2\right)
\end{equation*}
From these estimates it follows that, for any $\epsilon>0$, there is $M_{\epsilon}\geq 1$ for which
\begin{equation*}
|U(\lambda,\phi,f)|\leq \epsilon Q_{\Lambda^{\kappa}}(f)+M_{\epsilon}(1+R(\lambda))^{\kappa}\|f\|_2^2
\end{equation*}
for all $\lambda\in\mathbb{V}^*$, $\phi\in\mathcal{E}$ and $f\in C_0^{\infty}(\mathbb{V})$. By a similar argument, making use of Lemma \ref{lem:Scaling} and the fact that $W(\lambda,\phi,f)$ consists of ``lower order'' terms whose coefficients $b_{\alpha,\beta}$ are everywhere bounded, an analogous estimate can be made for $W(\lambda,\phi,f)$. From these estimates the lemma follows at once.
\end{proof}
\begin{proposition}\label{prop:SuperDuperSatisfiesHypothesis3}
Assume that Conditions \ref{cond:meas}-\ref{cond:constanttopsymbol} hold. Then $Q$ (and so $Q+C$) satisfies Hypothesis \ref{hyp:kappa}.
\end{proposition}
\begin{proof}
By virtue of Proposition \ref{prop:ESA} and Leibniz's rule, we see that
\begin{multline*}
\Dom(H^{\kappa}_{\lambda,\phi})=\{f\in L^2:e^{-\lambda(\phi)}f\in \Dom(H^{\kappa})\}\\
=\{f\in L^2:e^{-\lambda(\phi)}f\in W_{\mathbf{v}}^{2\kappa\mathbf{m},2}(\mathbb{V})\}=W_{\mathbf{v}}^{2\kappa\mathbf{m},2}(\mathbb{V})
\end{multline*}
for all $\lambda\in\mathbb{V}^*$ and $\phi\in\mathcal{E}$ where we fix $\kappa=\min\{n:\mu_{\Lambda}/n<1\}$. Consequently,
\begin{equation*}
\Dom(H^{\kappa}_{\lambda,\phi})=W_{\mathbf{v}}^{2\kappa\mathbf{m},2}(\mathbb{V})\subseteq W_{\mathbf{v}}^{\kappa\mathbf{m},2}(\mathbb{V})=\Dom(Q_{\Lambda^{\kappa}})
\end{equation*}
for all $\lambda\in\mathbb{V}^*$ and $\phi\in\mathcal{E}$. An appeal to the preceding lemma guarantees that, for any $\epsilon>0$, there is $M_{\epsilon}\geq 1$ for which
\begin{eqnarray*}
Q_{\Lambda^{\kappa}}(f)&=&\langle H_{\lambda,\phi}^{\kappa}f,f\rangle+Q_{\Lambda^{\kappa}}(f)-\langle H_{\lambda,\phi}^{\kappa}f,f\rangle\\
&\leq&|\langle H_{\lambda,\phi}^{\kappa}f,f\rangle|+|Q_{\Lambda^{\kappa}}(f)-\langle H_{\lambda,\phi}^{\kappa}f,f\rangle|\\
&\leq &M_\epsilon \left|\langle H_{\lambda,\phi}^{\kappa}f,f\rangle\right|+\epsilon Q_{\Lambda^{\kappa}}(f)+M_{\epsilon}(1+R(\lambda))^{\kappa}\|f\|_2^2
\end{eqnarray*}
for $\lambda\in\mathbb{V}^*$, $\phi\in\mathcal{E}$ and $f\in C_0^{\infty}(\mathbb{V})$. Equivalently,
\begin{equation*}
Q_{\Lambda^{\kappa}}(f)\leq\frac{M_\epsilon}{1-\epsilon}\left|\langle H_{\lambda,\phi}^{\kappa}f,f\rangle\right|+\frac{M_{\epsilon}}{1-\epsilon}(1+R(\lambda))^\kappa\|f\|_2^2
\end{equation*}
for all $\lambda\in\mathbb{V}^*$, $\phi\in\mathcal{E}$ and $f\in C_0^{\infty}(\mathbb{V})$. In view of Propositions \ref{prop:DirichletOperator} and \ref{prop:ESA}, $C_0^{\infty}(\mathbb{V})$ is a core for both $Q_{\Lambda^{\kappa}}$ and $H^{\kappa}$ and so it follows that the above estimate holds for all $\lambda\in\mathbb{V}^*$, $\phi\in\mathcal{E}$ and $f\in \Dom(H^{\kappa})=\Dom(H^{\kappa}_{\lambda,\phi})=W_{\mathbf{v}}^{2\kappa\mathbf{m},2}(\mathbb{V})$, as desired.
\end{proof}

\noindent In view of Propositions \ref{prop:SuperSatisfiesHypothesis1}, \ref{prop:SuperSatisfiesHypothesis2} and \ref{prop:SuperDuperSatisfiesHypothesis3}, an appeal to Theorem \ref{thm:Main} gives our final result for super-semi-elliptic-operators.

\begin{proposition}
Let $Q$ be a $\{2\mathbf{m},\mathbf{v}\}$-super-semi-elliptic form on $L^2(\mathbb{V})$ whose coefficients satisfy Conditions \ref{cond:meas}-\ref{cond:constanttopsymbol} with reference operator $\Lambda$ and associated self-adjoint super-semi-elliptic operator $H$. Let $R$ be the symbol and $\mu_{\Lambda}=|\mathbf{1}:2\mathbf{m}|$ be the homogeneous order of $\Lambda$, respectively. Then the semigroup $T_t=e^{-tH}$ has integral kernel $K_H:(0,\infty)\times\mathbb{V}\times\mathbb{V}\to\mathbb{C}$ satisfying
\begin{equation*}
|K_H(t,x,y)|\leq \frac{C}{t^{\mu_{\Lambda}}}\exp\left(-tMR^{\#}\left(\frac{x-y}{t}\right)+Mt\right)
\end{equation*}
for all $x,y\in\mathbb{V}$ and $t>0$ where $R^{\#}$ is the Legendre-Fenchel transform of $R$ and $C$ and $M$ are positive constants.
\end{proposition}

\begin{remark}
The above result is weaker than Theorem 5.1 of \cite{Randles2017} in that the latter treats semi-elliptic operators with H\"{o}lder continuous coefficients and allows for the operator's principal part to have variable coefficients. We have included this result because its proof is drastically different from that of Theorem 5.1 of \cite{Randles2017} and relies on the functional-analytic method of E. B. Davies \cite{Davies1995}, as we have adapted and presented in this article. It also illustrates that Davies' method can be extended into the realm in which $\mu_{\Lambda}\geq 1$ ( or $d\geq 2m$ for elliptic operators). As discussed in the following two remarks, we believe this result can be sharpened still while making use of our general theory presented in Theorem \ref{thm:Main}.
\end{remark}
\begin{remark} Condition \ref{cond:smooth}, a strong assumption, was used to establish that the powers of $H$ were sufficiently well behaved under perturbations thus establishing Proposition \ref{prop:SuperDuperSatisfiesHypothesis3}. It remains an open question as to what is the weakest smoothness assumption that can be made on the coefficients of $H$ to verify Hypothesis \ref{hyp:kappa}.
\end{remark}
\begin{remark} In checking the perturbative estimates in the proof of Proposition \ref{prop:SuperDuperSatisfiesHypothesis3}, it was useful to have $C_0^{\infty}(\mathbb{V})$ as a core for $\Dom(H^{\kappa})$. Under our assumptions, this fact relied on the formal expression for the $\kappa$th power of $H$, $H_0^{\kappa}$, to be essentially self-adjoint with closure $H^\kappa$. We ask: To what degree is this necessary?
\end{remark}

\appendix{}
\section{Appendix}

\subsection{The Legendre-Fenchel transform of a  positive-homogeneous polynomial}\label{Appendix:LF}

\noindent In this section, we state some result involving the Legendre-Fenchel transform of a positive-homogeneous polynomial relevant to our study. The results herein can be found in Section 3 of \cite{Randles2017} including their proofs. To this end, let $\mathbb{V}$ be a real $d$-dimensional vector space and $\mathbb{V}^*$ be its dual. Consider a positive-homogeneous polynomial $P:\mathbb{V}^*\to\mathbb{C}$ and set $R=\Re P$. The Legendre-Fenchel tranform of $R$ is the function $R^{\#}:\mathbb{V}\to\mathbb{R}$ defined by
\begin{equation*}
R^{\#}(x)=\sup_{\lambda\in\mathbb{V}^*}\{x(\lambda)-R(\lambda)\}
\end{equation*}
for $x\in\mathbb{V}$. The following proposition captures some useful facts about $R^{\#}$.

\begin{proposition}\label{prop:LegendreContinuousPositiveDefinite}
Let $P$ be a positive homogeneous polynomial with $R=\Re P$. Then $R^{\#}$ is continuous, positive-definite, and for any $E\in\Exp(P)$, $F=(I-E)^*\in\Exp(R^{\#})$. Further, given any polynomial $Q$ on $\mathbb{V}$ and $\epsilon>0$, we have
\begin{equation*}
Q(\cdot)e^{-\epsilon R^{\#}(\cdot)}\in L^\infty(\mathbb{V})\cap L^1(\mathbb{V})
\end{equation*}
and so, in particular, $\lim_{x\to\infty} R^{\#}(x)=\infty$.
\end{proposition}

\subsection{One-parameter contracting groups}

\noindent In what follows, $W$ is a $d$-dimensional real vector space with a norm $|\cdot|$; the corresponding operator norm on $\Gl(W)$ is denoted by $\|\cdot\|$. Of course, since everything is finite-dimensional, the usual topologies on $W$ and $\Gl(W)$ are insensitive to the specific choice of norms. Also, we say that two real-valued functions $f$ and $g$ on $W$ are comparable if, for some positive constant $C$, $C^{-1}f(w)\leq g(w)\leq C f(w)$ for all $w\in W$; in this case we write $f\asymp g$.
 
\begin{definition}
Let $\{T_t\}_{t>0}\subseteq \Gl(W)$ be a continuous one-parameter group. $\{T_t\}$ is said to be contracting if
\begin{equation*}
\lim_{t\rightarrow 0}\|T_t\|=0.
\end{equation*}
\end{definition}
\noindent We easily observe that, for any diagonalizable $E\in\End(W)$ with strictly positive spectrum, the corresponding one-parameter group $\{t^E\}_{t>0}$ is contracting. Indeed, if there exists a basis $\mathbf{w}=\{w_1,w_2,\dots,w_d\}$ of $W$ and a collection of positive numbers $\lambda_1,\lambda_2,\dots,\lambda_d$ for which $Ew_k=\lambda_kw_k$ for $k=1,2,\dots,d$, then the one parameter group $\{t^E\}_{t>0}$ has $t^{E}w_k=t^{\lambda_k}w_k$ for $k=1,2,\dots,d$ and $t>0$. It then follows immediately that $\{t^E\}$ is contracting. 

\begin{proposition}\label{prop:ComparePoly}
Let $Q$ and $R$ be continuous real-valued functions on $W$. If $R(w)>0$ for all $w\neq 0$ and there exists $E\in\Exp(Q)\cap\Exp(R)$ for which $\{t^E\}$ is contracting, then, for some positive constant $C$, $Q(w)\leq C R(w)$ for all $w\in W$. If additionally $Q(w)>0$ for all $w\neq 0$, then $Q\asymp R$.
\end{proposition}

\begin{proof}
Let $S$ denote the unit sphere in $W$ and observe that
\begin{equation*}
\sup_{w\in S}\frac{Q(w)}{R(w)}=:C<\infty
\end{equation*}
because $Q$ and $R$ are continuous and $R$ is non-zero on $S$. Now, for any non-zero $w\in W$, the fact that $t^E$ is contracting implies that $t^Ew\in S$ for some $t>0$ by virtue of the intermediate value theorem. Therefore, $Q(w)=Q(t^Ew)/t\leq CR(t^E w)/t=CR(w)$. In view of the continuity of $Q$ and $R$, this inequality must hold for all $w\in W$. When additionally $Q(w)>0$ for all non-zero $w$, the conclusion that $Q\asymp R$ is obtained by reversing the roles of $Q$ and $R$ in the preceding argument. 
\end{proof}

\begin{corollary}\label{cor:MovingConstants} Let $\Lambda$ be a positive-homogeneous operator on $\mathbb{V}$ with symbol $P$ and let $R^{\#}$ be the Legendre-Fenchel transform of $R=\Re P$. Then, for any positive constant $M$, $R^{\#}\asymp (MR)^{\#}$.
\end{corollary}
\begin{proof}
By virtue of Proposition \ref{prop:OperatorRepresentation}, let $\mathbf{m}\in\mathbb{N}_+^d$ and $\mathbf{v}$ be a basis for $\mathbb{V}$ and for which $E_{\mathbf{v}}^{2\mathbf{m}}\in \Exp(\Lambda)$. In view of Proposition \ref{prop:LegendreContinuousPositiveDefinite}, $R^{\#}$ and $(MR)^{\#}$ are both continuous, positive-definite and have $I-E_{\mathbf{v}}^{2\mathbf{m}}\in \Exp(R^{\#})\cap \Exp((MR)^{\#})$. In view of \eqref{eq:DefofE}, it is easily verified that $I-E_{\mathbf{v}}^{2\mathbf{m}}=E_\mathbf{v}^\omega$ where
\begin{equation}\label{eq:DefOfOmega}
\omega:=\left(\frac{2m_1}{2m_1-1},\frac{2m_2}{2m_2-1},\dots\frac{2m_d}{2m_d-1}\right)\in \mathbb{R}_+^d
\end{equation}
and so it follows that $\{t^{E_{\mathbf{v}}^{\omega}}\}$ is contracting. The corollary now follows directly from Proposition \ref{prop:ComparePoly}.
\end{proof}

\begin{lemma}\label{lem:Scaling}
Let $P$ be a positive-homogeneous polynomial on $W$ and let $\mathbf{n}=2\mathbf{m}\in\mathbb{N}_+^d$ and $\mathbf{w}$ be a basis for $W$ for which the conclusion of Proposition \ref{prop:OperatorRepresentation} holds. Let $\kappa\geq 1$ be an integer, set $R=\Re P$, and assume the notation of \eqref{eq:Monomial} in which $\xi^{\alpha}=\xi_{\mathbf{w}}^\alpha$ for each multi-index $\alpha$.
\begin{enumerate}
\item\label{item:Scaling1} Let $\alpha$ be a multi-index for which $|\alpha:2\mathbf{m}|\leq 1$.
\begin{enumerate}
\item There is a constant $M\geq 1$ for which
\begin{equation*}
|\xi^{\alpha}|\leq MR(\xi)+M
\end{equation*}
for all $\xi\in W$. 
\item If additionally $|\alpha:2\mathbf{m}|<1$, then, for any $\epsilon>0$, there is a constant $M_{\epsilon}\geq 1$ for which
\begin{equation*}
|\xi^{\alpha}|\leq \epsilon R(\xi)+M_{\epsilon}
\end{equation*}
for all $\xi\in W$.
\end{enumerate}
\item\label{item:Scaling2} Let $\alpha$ and $\beta$ be multi-indices for which $|\alpha+\beta:2\mathbf{m}|\leq \kappa$.
\begin{enumerate}
\item There are constants $M,M'\geq 1$ for which
\begin{equation*}
|\xi^{\alpha}\lambda^{\beta}|\leq MR(\xi)^{\kappa}+M'(R(\lambda)+1)^\kappa
\end{equation*}
for all $\xi,\lambda\in W$.
\item If additionally $|\alpha:2\mathbf{m}|<\kappa$, then, for any $\epsilon>0$, there is a constant $M_{\epsilon}\geq 1$ for which
\begin{equation*}
|\xi^{\alpha}\lambda^{\beta}|\leq \epsilon R(\xi)^{\kappa}+M_{\epsilon}(R(\lambda)+1)^{\kappa}
\end{equation*}
for all $\xi,\lambda\in W$.
\end{enumerate}
\end{enumerate}
\end{lemma}

\begin{proof}
We first note that Item \ref{item:Scaling1} can be seen as a consequence of Item \ref{item:Scaling2} by considering $\lambda=\phi_{\mathbf{w}}^{-1}(1,1,\dots,1)$ and $\kappa=1$. For the first assertion of Item \ref{item:Scaling2}, we assume that $|\alpha:2\mathbf{m}|+|\beta:2\mathbf{m}|=|\alpha+\beta:2\mathbf{m}|\leq \kappa$ and consider the contracting group $\{t^{E\oplus E}\}=\{t^{E}\oplus t^{E}\}$ on $W\oplus W$ where $E=E_{\mathbf{w}}^{2\mathbf{m}}\in \End(W)$. Because $R$ is a positive-definite polynomial, $R(\xi)^{\kappa}+R(\lambda)^\kappa$ is continuous and positive-definite on $W\oplus W$. Let $|\cdot |$ be a norm on $W\oplus W$ and respectively denote by $B$ and $S$ the corresponding unit ball and unit sphere in this norm. Observe that
\begin{equation*}
M:=\sup_{(\xi,\lambda)\in S}\frac{|\xi^{\alpha}\lambda^{\beta}|}{R(\xi)^{\kappa}+R(\lambda)^{\kappa}}<\infty.
\end{equation*}
Given any $(\xi,\lambda)\in W\oplus W\setminus B$, because $\{t^{E\oplus E}\}$ is contracting, it follows from the intermediate value theorem that, for some $t\geq 1$, $t^{-(E\oplus E)}(\xi,\lambda)=(t^{-E}\xi,t^{-E}\lambda)\in S$ and therefore
\begin{eqnarray*}
|\xi^{\alpha}\lambda^{\beta}|&=& t^{(|\alpha:2\mathbf{m}|+|\beta:2\mathbf{m}|)}|(t^{-E}\xi)^{\alpha}(t^{-E}\lambda)^{\beta}|\\
&\leq & t^{|\alpha+\beta:2\mathbf{m}|}M\left(R(t^{-E}\xi)^\kappa+R(t^{-E}\lambda)^{\kappa}\right)\\
&\leq &t^{\kappa}M((t^{-1}R(\xi))^{\kappa}+(t^{-1}R(\lambda))^{\kappa})\\
&\leq &M(R(\xi)^{\kappa}+R(\lambda)^{\kappa}).
\end{eqnarray*}
Upon noting that $|\xi^{\alpha}\lambda^{\beta}|$ is bounded as $(\xi,\lambda)$ varies over the compact set $B$, there exists $M'\geq M$ for which $|\xi^{\alpha}\lambda^\beta|\leq M'$ for all $(\xi,\lambda)\in B$. Putting these estimates together, we obtain
\begin{equation*}
|\xi^{\alpha}\lambda^{\beta}|\leq M (R(\xi)^{\kappa}+R(\lambda)^{\kappa})+M'\leq MR(\xi)+M'(R(\xi)+1)^{\kappa}
\end{equation*}
for all $(\xi,\lambda)\in W\oplus W$. 

To verify our final assertion, we appeal to the preceding estimate to see that
\begin{eqnarray*}
|\xi^{\alpha}\lambda^{\beta}|&=&t^{-|\alpha:2\mathbf{m}|}|(t^{E}\xi)^{\alpha}\lambda^{\beta}|\\
&\leq& t^{-|\alpha:2\mathbf{m}|}\left(MR(t^{E}\xi)^{\kappa}+M'(R(\lambda)+1)^{\kappa}\right)\\
&\leq &Mt^{(\kappa-|\alpha:2\mathbf{m}|)}R(\xi)^{\kappa}+M't^{-|\alpha:2\mathbf{m}|}(R(\lambda)+1)^{\kappa}
\end{eqnarray*}
for all $\xi,\lambda\in W$ and $t>0$. Upon noting that $|\alpha:2\mathbf{m}|<\kappa$, the desired result follows by choosing $t$ sufficiently small.
\end{proof}

\section*{\large Acknowledgment.}

The second author was supported by the National Science Foundation under Grant No. DMS-1707589. The authors would like to thank Ms. Darby Beaulieu and Dr. Christopher Nosala for proofreading sections of this article.

\noindent Evan Randles: Department of Mathematics \& Statistics, Colby College, 5834 Mayflower Hill, Waterville, ME 04901.
\newline E-mail: evan.randles@colby.edu\\

\noindent Laurent Saloff-Coste: Department of Mathematics, Cornell University, Ithaca NY 14853.
\newline E-mail: lsc@math.cornell.edu

\end{document}